\documentclass{amsart}
\usepackage[latin1]{inputenc}
\usepackage{subfigure}
\usepackage{color}
\usepackage{amsmath}
\usepackage{amsthm}
\usepackage{amstext}
\usepackage{amssymb}
\usepackage{amsfonts}
\usepackage{graphicx}
\usepackage{young}
\usepackage{multicol}
\usepackage{mathrsfs}
\usepackage{stmaryrd}
\usepackage{bbm}
\usepackage[all]{xy}
\usepackage{mathabx}
\usepackage{tikz}
\usepackage{tikz-cd}
\usepackage{mathtools}
\usepackage{tabularx}
\usepackage{array}

\DeclareMathOperator{\rep}{\mathrm{rep}}
\DeclareMathOperator{\Ext}{\mathrm{Ext}}
\DeclareMathOperator{\Hom}{\mathrm{Hom}}

\DeclareMathOperator{\ZZ}{\mathbb{Z}}

\DeclareMathOperator{\AAA}{\mathbb{A}}

\DeclareMathOperator{\cD}{\mathcal{D}}
\DeclareMathOperator{\bc}{\textbf{c}}

\theoremstyle{plain}
\theoremstyle{definition}
\newtheorem{theorem}{Theorem}[section]

\newtheorem{remark}[theorem]{Remark}
\newtheorem{lemma}[theorem]{Lemma}
\newtheorem{definition}[theorem]{Definition}
\newtheorem{notation}[theorem]{Notation}

\newtheorem{example}[theorem]{Example}

\newtheorem{corollary}[theorem]{Corollary}

\newtheorem*{lemma*intro}{Lemma \ref{maintechlemma}}

\DeclareMathAlphabet{\mathpzc}{OT1}{pzc}{m}{it}

\title{Combinatorics of exceptional sequences in type A}

\author{Alexander Garver}
\address{Laboratoire de Combinatoire et d'Informatique Math\'ematique, Universit\'e du Qu\'ebec \`a Montr\'eal, Montr\'eal, QC H3C 3P8, Canada}
\email{alexander.garver@lacim.ca}

\author{Kiyoshi Igusa}
\address{Department of Mathematics,
Brandeis University, Waltham, MA 02454, USA}
\email{igusa@brandeis.edu}

\author{Jacob P. Matherne}
\address{Department of Mathematics and Statistics,
University of Massachusetts Amherst, Amherst, MA 01003, USA}
\email{matherne@math.umass.edu}

\author{Jonah Ostroff}
\address{Department of Mathematics, University of Washington, Seattle, WA 98195, USA}
\email{ostroff@math.washington.edu}

\thanks{The first author was supported by a Research Training Group, RTG grant DMS-1148634.}
\thanks{The second author was supported by National Security Agency Grant H98230-13-1-0247.}
\thanks{The third author was supported by a Graduate Assistance in Areas of National Need Fellowship, GAANN grant P200A120001, and an LSU Dissertation Year Fellowship.}

\begin{document}

\begin{abstract}
Exceptional sequences are certain ordered sequences of quiver representations. We introduce a class of objects called strand diagrams and use this model to classify exceptional sequences of representations of a quiver whose underlying graph is a type $\mathbb{A}_n$ Dynkin diagram. We also use variations of this model to classify \textbf{c}-matrices of such quivers, to interpret exceptional sequences as linear extensions of posets, and to give a simple bijection between exceptional sequences and certain chains in the lattice of noncrossing partitions. This work is an extension of the classification of exceptional sequences for the linearly-ordered quiver by the first and third authors.
\end{abstract}

\maketitle
\tableofcontents

\section{Introduction}
Exceptional sequences are certain sequences of quiver representations with strong homological properties.  They were introduced in \cite{gr87} to study exceptional vector bundles on $\mathbb{P}^2$, and more recently, Crawley-Boevey showed that the braid group acts transitively on the set of complete exceptional sequences (exceptional sequences of maximal length) \cite{c93}.  This result was generalized to hereditary Artin algebras by Ringel \cite{r94}.  Since that time, Meltzer has also studied exceptional sequences for weighted projective lines \cite{m04}, and Araya for Cohen-Macaulay modules over one\textcolor{green}{-}dimensional graded Gorenstein rings with a simple singularity \cite{a99}. Exceptional sequences have been shown to be related to many other areas of mathematics since their invention:
\begin{itemize}
\item chains in the lattice of noncrossing partitions \cite{b03,hk13,it09},
\item $\textbf{c}$-matrices and cluster algebras \cite{st13},
\item factorizations of Coxeter elements \cite{is10}, and
\item $t$-structures and derived categories \cite{bez03, bk89, r90}.
\end{itemize}

\noindent Despite their ubiquity, very little work has been done to concretely describe exceptional sequences, even for path algebras of Dynkin quivers \cite{a13, gm1}.  In this paper, we give a concrete description of exceptional sequences for type $\AAA_n$ quivers of any orientation.  This work extends and elaborates on a classification of exceptional sequences for the linearly-ordered quiver obtained in \cite{gm1} by the first and third authors.

Exceptional sequences are composed of indecomposable representations which have a particularly nice description.  For a quiver $Q$ of type $\mathbb{A}_n$, the indecomposable representations are completely determined by their dimension vectors, which are of the form $(0,\ldots,0,1,\ldots,1,0,\ldots,0) \ \in \mathbb{Z}_{\ge 0}^n$.  Let us denote such a representation by $X_{i,j}^{\epsilon}$, where $\epsilon$ is a vector that keeps track of the orientation of the quiver, and $i+1$ and $j$ are the positions where the string of $1$'s begins and ends, respectively.

This simple description allows us to view exceptional sequences as combinatorial objects.  Define a map $\Phi_\epsilon$ which associates to each indecomposable representation $X_{i,j}^{\epsilon}$ a curve $\Phi_\epsilon(X_{i,j}^{\epsilon})$ connecting two of $n+1$ points in $\mathbb{R}^2$. We will refer to such curves as \textbf{strands}.\footnote{The curves $\Phi_\epsilon(X_{i,j}^{\epsilon})$ will have some additional topological conditions that we omit here.}

\begin{figure}[h]
\includegraphics[scale=1]{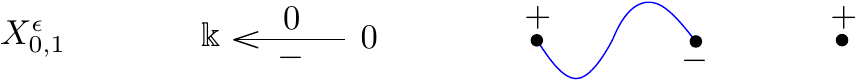}
\caption{An example of the indecomposable representation $X_{0,1}^\epsilon$ on a type $\mathbb{A}_2$ quiver and the corresponding strand $\Phi_\epsilon(X_{0,1}^{\epsilon})$}
\label{strex}
\end{figure}

\noindent As exceptional sequences are collections of representations, the map $\Phi_\epsilon$ allows one to regard them as collections of strands.  The following lemma is the foundation for all of our results in this paper (it characterizes the homological data encoded by a pair of strands and thus by a pair of representations). An exceptional sequence is completely determined by its \textbf{exceptional pairs} (i.e., its subsequences of length 2). Thus, Lemma~\ref{maintechlemma}, which we now state, allows us to completely classify exceptional sequences using strand diagrams. 

\begin{lemma*intro}
Let $Q_\epsilon$ be a type $\mathbb{A}$ Dynkin quiver. Fix two distinct indecomposable representations $U, V \in \text{ind}(\text{rep}_\Bbbk(Q_\epsilon))$.

$\begin{array}{rll}
a) & \text{The strands $\Phi_{\epsilon}(U)$ and $\Phi_{\epsilon}(V)$ intersect nontrivially if and only if neither $(U,V)$ nor $(V,U)$ are}\\
& \text{exceptional pairs.}\\
b) & \text{The strand $\Phi_{\epsilon}(U)$ is clockwise from $\Phi_{\epsilon}(V)$  if and only if $(U,V)$ is an exceptional pair and $(V,U)$}\\
& \text{is not an exceptional pair.}\\
c) & \text{The strands $\Phi_{\epsilon}(U)$ and $\Phi_{\epsilon}(V)$ do not intersect at any of their endpoints and they do not intersect}\\
& \text{nontrivially if and only if $(U, V)$ and $(V, U)$ are both exceptional pairs.}
\end{array}$
\end{lemma*intro}

The paper is organized in the following way.  In Section \ref{sec:prelim}, we give the preliminaries on quivers and their representations which are needed for the rest of the paper.

In Section \ref{sec:strands}, we introduce strand diagrams. Later, we decorate our strand diagrams with strand-labelings and oriented edges so that they can keep track of both the ordering of the representations in a complete exceptional sequence as well as the signs of the rows in the $\bc$-matrix it came from.  While unlabeled diagrams classify complete exceptional collections (Theorem \ref{ECbij}), we show that the new decorated diagrams classify exceptional sequences (Theorem \ref{ESbij}).  Although Lemma \ref{maintechlemma} is the main tool that allows us to obtain these results, we delay its proof until Section \ref{firstproofs}.

The work of Speyer and Thomas (see \cite{st13}) allows one to obtain complete exceptional sequences from $\bc$-matrices.  In \cite{onawfr13}, the number of complete exceptional sequences in type $\AAA_n$ is given, and there are more of these than there are $\bc$-matrices.  Thus, it is natural to ask exactly which $\bc$-matrices appear as strand diagrams.  By establishing a bijection between the mixed cobinary trees of Igusa and Ostroff \cite{io13} and a certain subcollection of strand diagrams, we give an answer to this question in Section~\ref{sec:mct}.

In Section \ref{sec:pos}, we ask how many complete exceptional sequences can be formed using the representations in a complete exceptional collection. It turns out that two complete exceptional sequences can be formed in this way if they have the same underlying strand diagram without strand labels. We interpret this number as the number of linear extensions of a poset determined by the strand diagram of the complete exceptional collection. This also gives an interpretation of complete exceptional sequences as linear extensions.

In Section \ref{sec:app}, we give several applications of the theory in type $\AAA$, including combinatorial proofs that two reddening sequences produce isomorphic ice quivers (see \cite{k12} for a general proof in all types using deep category-theoretic techniques) and that there is a bijection between exceptional sequences and saturated chains in the lattice of noncrossing partitions.

{\bf Acknowledgements.~}
A. Garver and J. Matherne gained helpful insight through conversations with E. Barnard, J. Geiger, M. Kulkarni, G. Muller, G. Musiker, D. Rupel, D. Speyer, and G. Todorov.  A. Garver and J. Matherne also thank the 2014 AMS Mathematics Research Communities program for giving us a stimulating place to work. The authors thank an anonymous referee for comments that helped to improve the exposition.

\section{Preliminaries}\label{sec:prelim}

We will be interested in the connection between exceptional sequences and the \textbf{c}-matrices of an acyclic quiver $Q$ so we begin by defining these. After that we review the basic terminology of quiver representations and exceptional sequences, which serve as the starting point in our study of exceptional sequences. We conclude this section by explaining the notation we will use to discuss exceptional representations of quivers that are orientations of a type $\mathbb{A}_n$ Dynkin diagram.

\subsection{Quiver mutation}\label{subsec:quivers}
A \textbf{quiver} $Q$ is a directed graph without loops or 2-cycles. In other words, $Q$ is a 4-tuple $(Q_0,Q_1,s,t)$, where $Q_0 = [m] := \{1,2, \ldots, m\}$ is a set of \textbf{vertices}, $Q_1$ is a set of \textbf{arrows}, and two functions $s, t:Q_1 \to Q_0$ defined so that for every $a \in Q_1$, we have $s(a) \xrightarrow{a} t(a)$. An \textbf{ice quiver} is a pair $(Q,F)$ with $Q$ a quiver and $F \subset Q_0$ \textbf{frozen vertices} with the additional restriction that any $i,j \in F$ have no arrows of $Q$ connecting them. We refer to the elements of $Q_0\backslash F$ as \textbf{mutable vertices}. By convention, we assume $Q_0\backslash F = [n]$ and $F = [n+1,m] := \{n+1, n+2, \ldots, m\}.$ Any quiver $Q$ can be regarded as an ice quiver by setting $Q = (Q, \emptyset)$.

The {\bf mutation} of an ice quiver $(Q,F)$ at mutable vertex $k$, denoted $\mu_k$, produces a new ice quiver $(\mu_kQ,F)$ by the three step process:

(1) For every $2$-path $i \to k \to j$ in $Q$, adjoin a new arrow $i \to j$.

(2) Reverse the direction of all arrows incident to $k$ in $Q$.

(3) Remove all 2-cycles in the resulting quiver as well as all of the arrows between frozen vertices.

\noindent We show an example of mutation below, depicting the mutable (resp. frozen) vertices in black (resp. blue).
\[
\begin{array}{c c c c c c c c c}
\raisebox{.35in}{$(Q,F)$} & \raisebox{.35in}{=} & {\includegraphics[scale = .7]{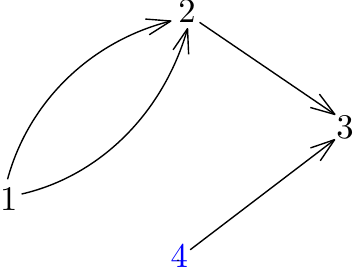}} & \raisebox{.35in}{$\stackrel{\mu_2}{\longmapsto}$} & {\includegraphics[scale = .7]{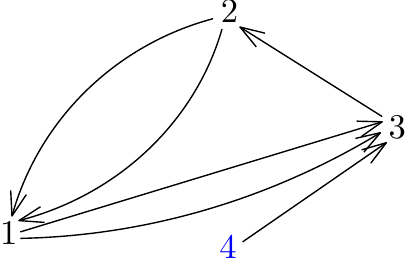}} & \raisebox{.35in}{=} & \raisebox{.35in}{$(\mu_2Q,F)$}
\end{array}
\]

The information of an ice quiver can be equivalently described by its (skew-symmetric) \textbf{exchange matrix}. Given $(Q,F),$ we define $B = B_{(Q,F)} = (b_{ij}) \in \mathbb{Z}^{n\times m} := \{n \times m \text{ integer matrices}\}$ by $b_{ij} := \#\{i \stackrel{a}{\to} j \in Q_1\} - \#\{j \stackrel{a}{\to} i \in Q_1\}.$ Furthermore, ice quiver mutation can equivalently be defined  as \textbf{matrix mutation} of the corresponding exchange matrix. Given an exchange matrix $B \in \mathbb{Z}^{n\times m}$, the \textbf{mutation} of $B$ at $k \in [n]$, also denoted $\mu_k$, produces a new exchange matrix $\mu_k(B) = (b^\prime_{ij})$ with entries
\[
b^\prime_{ij} := \left\{\begin{array}{ccl}
-b_{ij} & : & \text{if $i=k$ or $j=k$} \\
b_{ij} + \frac{|b_{ik}|b_{kj}+ b_{ik}|b_{kj}|}{2} & : & \text{otherwise.}
\end{array}\right.
\]
\begin{flushleft}For example, the mutation of the ice quiver above (here $m=4$ and $n=3$) translates into the following matrix mutation. Note that mutation of matrices {(and of ice quivers)} is an involution (i.e. $\mu_k\mu_k(B) = B$).\end{flushleft}
\[
\begin{array}{c c c c c c c c c c}
B_{(Q,F)} & = & \left[\begin{array}{c c c | r}
0 & 2 & 0 & 0 \\
-2 & 0 & 1 & 0\\
0 & -1 & 0 & -1\\
\end{array}\right]
& \stackrel{\mu_2}{\longmapsto} &
\left[\begin{array}{c c c | r}
0 & -2 & 2 & 0 \\
2 & 0 & -1 & 0\\
-2 & 1 & 0 & -1\\
\end{array}\right] 
& = & B_{(\mu_2Q,F)}.
\end{array}
\]

{Given a quiver $Q$, we define its \textbf{framed} (resp. \textbf{coframed}) quiver to be the ice quiver $\widehat{Q}$ (resp. $\widecheck{Q}$) where $\widehat{Q}_0\ (= \widecheck{Q}_0) := Q_0 \sqcup [n+1, 2n]$, $F = [n+1, 2n]$, and $\widehat{Q}_1 := Q_1 \sqcup \{i \to n+i: i \in [n]\}$ (resp. $\widecheck{Q}_1 := Q_1 \sqcup \{n+i \to i: i \in [n]\}$).}  Now given $\widehat{Q}$ we define the \textbf{exchange tree} of $\widehat{Q}$, denoted $ET(\widehat{Q})$, to be the (a priori infinite) graph whose vertices are quivers obtained from $\widehat{Q}$ by a finite sequence of mutations and with two vertices connected by an edge if and only if the corresponding quivers are obtained from each other by a single mutation. Similarly, define the \textbf{exchange graph} of $\widehat{Q}$, denoted $EG(\widehat{Q})$, to be the quotient of $ET(\widehat{Q})$ where two vertices are identified if and only if there is a \textbf{frozen isomorphism} of the corresponding quivers (i.e. an isomorphism that fixes the frozen vertices). Such an isomorphism is equivalent to a simultaneous permutation of the rows and columns of the corresponding exchange matrices.

Given $\widehat{Q}$, we define the \textbf{c}-\textbf{matrix} $C(n) = C_R(n)$ (resp. $C = C_R$) of $R \in ET(\widehat{Q})$ (resp. $R \in EG(\widehat{Q})$)  to be the submatrix of $B_R$ where $C(n) := (b_{ij})_{i \in [n], j \in [n+1, 2n]}$ (resp. $C := (b_{ij})_{i \in [n], j \in [n+1,2n]}$). We let \textbf{c}-mat($Q$) $:= \{C_R: R \in EG(\widehat{Q})\}$. By definition, $B_R$ (resp. $C$) is only defined up to simultaneous permutations of its rows and first $n$ columns (resp. up to permutations of its rows) for any $R \in EG(\widehat{Q})$.

A row vector of a \textbf{c}-matrix, $\overrightarrow{c}$, is known as a \textbf{c}-\textbf{vector}. The celebrated theorem of Derksen, Weyman, and Zelevinsky \cite[Theorem 1.7]{dwz10}, known as {sign-coherence} of $\textbf{c}$-vectors, states that for any $R \in ET(\widehat{Q})$ and $i \in [n]$ the \textbf{c}-vector $\overrightarrow{c_i}$ is a nonzero element of $\mathbb{Z}_{\ge 0}^n$ or $\mathbb{Z}_{\le0}^n$. In the former case, we say a \textbf{c}-vector is  \textbf{positive}, and in the latter case, we say a \textbf{c}-vector is \textbf{negative}.

\subsection{Representations of quivers}

A \textbf{representation} $V = ((V_i)_{i \in Q_0}, (\varphi_a)_{a \in Q_1})$ of a quiver $Q$  is an assignment of a finite dimensional $\Bbbk$-vector space $V_i$ to each vertex $i$ and a $\Bbbk$-linear map $\varphi_a: V_{s(a)} \rightarrow V_{t(a)}$ to each arrow $a$ where $\Bbbk$ is a field.  The \textbf{dimension vector} of $V$ is the vector $\underline{\dim}(V):=(\dim V_i)_{i\in Q_0}$.  The \textbf{support} of $V$ is the set $\text{supp}(V) := \{i\in Q_0 : V_i \neq 0\}$.
Here is an example of a representation, with $\underline{\dim}(V) = (3,3,2)$, of the \textbf{mutable part} of the quiver depicted in Section \ref{subsec:quivers}.
\[
{\includegraphics[scale = .8]{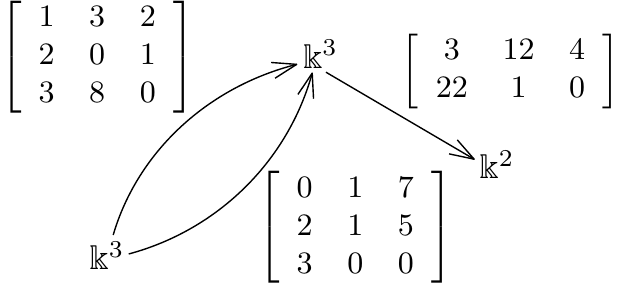}}
\]
Let $V = ((V_i)_{i \in Q_0}, (\varphi_a)_{a \in Q_1})$ and $W  = ((W_i)_{i \in Q_0}, (\varrho_a)_{a \in Q_1})$ be two representations of a quiver $Q$. A \textbf{morphism} $\theta : V \rightarrow W$ consists of a collection of linear maps $\theta_i : V_i \rightarrow W_i$ that are compatible with each of the linear maps in $V$ and $W$.  That is, for each arrow $a \in Q_1$, we have $\theta_{t(a)} \circ \varphi_a = \varrho_a \circ \theta_{s(a)}$.  An \textbf{isomorphism} of quiver representations is a morphism $\theta: V \to W$ where $\theta_i$ is a $\Bbbk$-vector space isomorphism for all $i \in Q_0$. We define $V \oplus W := ((V_i\oplus W_i)_{i \in Q_0}, (\varphi_a \oplus \varrho_a)_{a \in Q_1})$ to be the \textbf{direct sum} of $V$ and $W$. We say that a nonzero representation $V$ is \textbf{indecomposable} if it is not isomorphic to a direct sum of two nonzero representations. Note that the representations of a quiver $Q$ along with morphisms between them form an abelian category, denoted by $\text{rep}_{\Bbbk}(Q)$, with the indecomposable representations forming a full subcategory, denoted by $\text{ind}(\text{rep}_\Bbbk(Q))$.

We remark that representations of $Q$ can equivalently be regarded as modules over the \textbf{path algebra} $\Bbbk Q$. As such, one can define $\Ext_{\Bbbk Q}^s(V,W)$ for $s \ge 1$ and $\Hom_{\Bbbk Q}(V,W)$ for any representations $V$ and $W$, and $\Hom_{\Bbbk Q}(V,W)$ is isomorphic to the vector space of all morphisms $\theta:V\to W$.  We refer the reader to \cite{ass06} for more details on representations of quivers.

An \textbf{exceptional sequence} $\xi = (V_1,\ldots, V_k)$  ($k \le n:= \#Q_0$) is a sequence of \textbf{exceptional representations}\footnote{{In general, not all indecomposable representations are exceptional. For Dynkin quivers, it is well-known that a representation is exceptional if and only if it is indecomposable.} } $V_j$ of $Q$ (i.e. $V_j$ is indecomposable and $\Ext_{\Bbbk Q}^s(V_j,V_j) = 0$ for all $s\ge 1$) satisfying $\Hom_{\Bbbk Q}(V_j,V_i) = 0$ and $\Ext_{\Bbbk Q}^s(V_j,V_i) = 0$ if $i < j$ for all $s \ge 1$. We use the term \textbf{exceptional pair} to mean an exceptional sequence consisting of exactly two exceptional representations. We define an \textbf{exceptional collection} $\overline{\xi} = \{V_1,\ldots, V_k\}$ to be a \textit{set} of exceptional representations $V_j$ of $Q$ that can be ordered in such a way that they define an exceptional sequence. When $k = n$, we say $\xi$ (resp. $\overline{\xi}$) is a \textbf{complete exceptional sequence} (CES) (resp. \textbf{complete exceptional collection} (CEC)). 

The following result of Speyer and Thomas gives a beautiful connection between $\bc$-matrices of an acyclic quiver $Q$ and CESs.  It serves as motivation for our work. Before stating it we remark that for any $R \in ET(\widehat{Q})$ and any $i \in [n]$ where $Q$ is an acyclic quiver, the $\bc$-vector $\overrightarrow{c_i} = \overrightarrow{c_i}(R) = \pm \underline{\dim}(V_i)$ for some exceptional representation of $Q$ (see \cite{c12} or \cite{st13}). 

\begin{notation}
Let $\overrightarrow{c}$ be a \textbf{c}-vector of an acyclic quiver $Q$. Define 
$$\begin{array}{rcl}
|\overrightarrow{c}| & := & \left\{\begin{array}{rcl} \overrightarrow{c} & : & \text{if $\overrightarrow{c}$ is positive}\\ -\overrightarrow{c} & : & \text{if $\overrightarrow{c}$ is negative.} \end{array}\right.
\end{array}$$
\end{notation}

\begin{theorem}[\cite{st13}\label{st}]
Let $C \in \textbf{c}$-mat$(Q)$, let $\{\overrightarrow{c_i}\}_{i \in [n]}$ denote the \textbf{c}-vectors of $C$, and let $|\overrightarrow{c_i}| = \underline{\dim}(V_i)$ for some exceptional representation of $Q$. There exists a permutation $\sigma \in \mathfrak{S}_n$ such that  $(V_{\sigma(1)},...,V_{\sigma(n)})$ is a CES with the property that if there exist positive \textbf{c}-vectors in $C$, then there exists $k \in [n]$ such that $\overrightarrow{c_{\sigma(i)}}$ is positive if and only if $i \in [k,n]$, and $\Hom_{\Bbbk Q}(V_j,V_{j^\prime})=0$ for any $\overrightarrow{c_j}, \overrightarrow{c_{j^\prime}}$ that have the same sign.  Conversely, any set of $n$ vectors $\overrightarrow{c}_1,\ldots, \overrightarrow{c}_n$ having these properties defines a $\bc$-matrix whose row vectors are $\{\overrightarrow{c_i}\}_{i\in[n]}$.
\end{theorem}

\subsection{Quivers of type $\mathbb{A}_n$}

For the purposes of this paper, we will only be concerned with quivers of type $\AAA_n$. We say a quiver $Q$ is of \textbf{type} $\mathbb{A}_n$ if the underlying graph of $Q$ is a Dynkin diagram of type $\mathbb{A}_n$. By convention, two vertices $i$ and $j$ with $i < j$ in a type $\mathbb{A}_n$ quiver $Q$ are connected by an arrow if and only if $j = i+1$ and $i \in [n-1]$.

It will be convenient to denote a given type $\mathbb{A}_n$ quiver $Q$ using the notation $Q_\epsilon$, which we now define. Let $\epsilon = (\epsilon_0,\epsilon_1,\ldots, \epsilon_{n}) \in \{+,-\}^{n+1}$ and for $i \in [n-1]$ define $a_{i}^{\epsilon_i} \in Q_1$ by
$$\begin{array}{cccccccccccc}
a_{i}^{\epsilon_i} & := & \left\{\begin{array}{rcl} i\leftarrow i+1 &: & \epsilon_i = -\\ i\rightarrow i+1 & : & \epsilon_i = +. \end{array}\right. 
\end{array}$$
Then $Q_\epsilon := ((Q_\epsilon)_0 := [n], (Q_\epsilon)_1 := \{a_i^{\epsilon_i}\}_{i \in [n-1]}) = Q.$ One observes that the values of $\epsilon_0$ and $\epsilon_{n}$ do not affect $Q_\epsilon.$
\begin{example}
Let $n = 5$ and $\epsilon = (-,+, -, +,-,+)$ so that $Q_{\epsilon} = 1 \stackrel{a_1^+}{\longrightarrow} 2 \stackrel{a_2^-}{\longleftarrow} 3 \stackrel{a_3^+}{\longrightarrow} 4 \stackrel{a_4^-}{\longleftarrow} 5.$ Below we show its framed quiver $\widehat{Q}_\epsilon$.
\[\widehat{Q}_\epsilon = \begin{xy} 0;<1pt,0pt>:<0pt,-1pt>:: 
(0,10) *+{1} ="0",
(30,10) *+{2} ="1",
(60,10) *+{3} ="2",
(90,10) *+{4} ="3",
(120,10) *+{5} ="4",
(0,-20) *+{\textcolor{blue}{\text{$6$}}} ="5",
(30,-20) *+{\textcolor{blue}{\text{$7$}}} ="6",
(60, -20) *+{\textcolor{blue}{\text{$8$}}} = "9",
(90,-20) *+{\textcolor{blue}{\text{$9$}}} ="7",
(120,-20) *+{\textcolor{blue}{\text{$10$}}} ="8",
"0", {\ar"1"},
"0", {\ar"5"},
"2", {\ar"1"},
"1", {\ar"6"},
"2", {\ar"3"},
"4", {\ar"3"},
"3", {\ar"7"},
"4", {\ar"8"},
"2", {\ar"9"},
\end{xy}\]
\end{example}

Let $Q_\epsilon$ be given where $\epsilon = (\epsilon_0,\epsilon_1,\ldots, \epsilon_{n}) \in \{+,-\}^{n+1}.$ Let $i, j \in [0,n] := \{0,1,\ldots, n\}$ where $i < j$ and let $X^\epsilon_{i,j} = ((V_\ell)_{\ell \in (Q_\epsilon)_0}, (\varphi^{i,j}_a)_{a \in (Q_\epsilon)_1}) \in \text{rep}_{\Bbbk}(Q_\epsilon)$ be the indecomposable representation defined by

$$\begin{array}{cccccccccccc}
V_\ell & := & \left\{\begin{array}{rcl} \Bbbk &: & i+1 \le \ell \le j\\ 0 & : & \text{otherwise} \end{array}\right. & & & & \varphi^{i,j}_{a} & := & \left\{\begin{array}{rcl} 1 & : & a = a^{\epsilon_k}_k \text{ where } i+1 \le k \le j-1\\ 0 & :  & \text{otherwise.} \end{array}\right. 
\end{array}$$

\noindent The objects of $\text{ind}(\text{rep}_{\Bbbk}(Q_\epsilon))$ are those of the form $X^\epsilon_{i,j}$ where $0 \le i < j \le n$, up to isomorphism.

\begin{remark}\label{rem:hom0}
If $X^\epsilon_{i,j}$ and $X^\epsilon_{k,\ell}$ are distinct indecomposables of $\text{rep}_\Bbbk(Q_\epsilon)$, then $\text{Hom}_{\Bbbk Q_\epsilon}(X^\epsilon_{i,j}, X^\epsilon_{k,\ell}) = 0$ or $\text{Hom}_{\Bbbk Q_\epsilon}(X^\epsilon_{k,\ell}, X^\epsilon_{i,j}) = 0.$ This follows from the well-known fact that the Auslander--Reiten quiver of $\Bbbk Q_\epsilon$ is acyclic.
\end{remark}

\section{Strand diagrams}

In this section, we define three different types of combinatorial objects: strand diagrams, labeled strand diagrams, and oriented strand diagrams. We will use these objects to classify exceptional collections, exceptional sequences, and \textbf{c}-matrices of a given type $\mathbb{A}_n$ quiver $Q_\epsilon$, so we fix such a quiver $Q_\epsilon$.

\subsection{Exceptional sequences and strand diagrams}\label{sec:strands}

Let $\mathcal{S}_{n,\epsilon} \subset \mathbb{R}^2$ be a collection of $n+1$ points arranged in a horizontal line. We identify these points with $\epsilon_0, \epsilon_1, \ldots, \epsilon_{n}$ where  $\epsilon_j$ appears to the right of $\epsilon_i$ for any $i,j \in [0,n] := \{0,1,2, \ldots, n\}$ where $i < j$. Using this identification, we can write $\epsilon_i = (x_i,y_i)\in \mathbb{R}^2$.

\begin{definition}
{Let $i,j \in [0,n]$ where $i\neq j$. A \textbf{strand} $c(i,j)$ on $\mathcal{S}_{n,\epsilon}$ is an isotopy class of simple curves in $\mathbb{R}^2$ where any $\gamma \in c(i,j)$ satisfies:} 

\begin{itemize}
\item[$a)$] the endpoints of {$\gamma$} are $\epsilon_i$ and $\epsilon_j$,
\item[$b)$] as a subset of $\mathbb{R}^2$, ${\gamma} \subset \{(x,y) \in \mathbb{R}^2: x_i \le x \le x_j \}\backslash\{\epsilon_{i+1}, \epsilon_{i+2}, \ldots, \epsilon_{j-1}\}$,
\item[$c)$] if $k \in [0,n]$ satisfies $i \le k \le j$ and $\epsilon_k = +$ (resp. $\epsilon_k = -$), then {$\gamma$} is locally below (resp. locally above) $\epsilon_k$. By locally below (resp. locally above) $\epsilon_k$, we mean that for a given parameterization of $\gamma = (\gamma^{(1)}, \gamma^{(2)}) :[0,1]\to \mathbb{R}^2$ there exists $\delta > 0$ such that $\gamma$ satisfies $\gamma^{(2)}(t) \le y_k$ (resp. $\gamma^{(2)}(t) \ge y_k$) for all $t \in [0,1]$ where $\gamma(t) \in \{p \in \mathbb{R}^2: \ \text{dist}(p, \epsilon_{k}) < \delta\}$. 
\end{itemize}

\noindent There is a natural bijection  $\Phi_{\epsilon}$ from the objects of $\text{ind}(\text{rep}_\Bbbk(Q_\epsilon))$ to the set of strands on $\mathcal{S}_{n,\epsilon}$ given by $\Phi_{\epsilon}(X^\epsilon_{i,j}) := c(i,j).$ 
\end{definition}

\begin{remark}
It is clear that any strand $c(i,j)$ can be represented by a \textbf{monotone curve} $\gamma \in c(i, j).$ That is, there exists a curve $\gamma \in c(i,j)$ with a parameterization $\gamma = (\gamma^{(1)}, \gamma^{(2)}): [0,1] \to \mathbb{R}^2$ such that if $t, s \in [0,1]$ and $t < s$, then $\gamma^{(1)}(t) < \gamma^{(1)}(s)$. 
\end{remark}

We say that two strands $c(i_1, j_1)$ and $c(i_2,j_2)$ \textbf{intersect nontrivially} if any two curves $\gamma_\ell \in c(i_\ell, j_\ell)$ with $\ell \in \{1,2\}$ have at least one transversal crossing. Otherwise, we say that $c(i_1, j_1)$ and $c(i_2,j_2)$ \textbf{do not intersect nontrivially}. For example, $c(1,3),c(2,4)$ intersect nontrivially if and only if $\epsilon_2=\epsilon_3$. Additionally, we say that $c(i_2,j_2)$ is \textbf{clockwise} from $c(i_1,j_1)$ (or, equivalently, $c(i_1,j_1)$ is \textbf{counterclockwise} from $c(i_2,j_2)$) if and only if some ${\gamma_1} \in c(i_1,j_1)$ and ${\gamma_2} \in c(i_2,j_2)$ share an endpoint $\epsilon_k$ and locally appear in one of the following six configurations up to isotopy. It follows from Lemma~\ref{nocommonendpt} below that if a strand is clockwise or counterclockwise from another strand, then the two do not intersect nontrivially.
$$\includegraphics[scale=1.5]{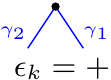} \ \ \ \ \ \ \ \ \ \includegraphics[scale=1.5]{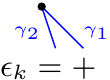} \ \ \ \ \ \ \ \ \ \includegraphics[scale=1.5]{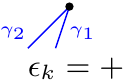} \ \ \ \ \ \ \ \ \  \includegraphics[scale=1.5]{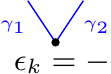} \ \ \ \ \ \ \ \ \ \includegraphics[scale=1.5]{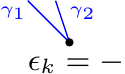} \ \ \ \ \ \ \ \ \ \includegraphics[scale=1.5]{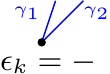}.$$


A given collection of strands $d = \{c(i_\ell, j_\ell)\}_{\ell \in [k]}$ with $k \le n$, naturally defines a graph with vertex set $\{\epsilon_0, \ldots, \epsilon_n\}$ and edge set $\{\{\epsilon_s, \epsilon_t\}: \ c(s, t) \in d\}$. We refer to this graph as the \textbf{graph determined by} $d$.

\begin{definition}
A \textbf{strand diagram} $d = \{c(i_\ell, j_\ell)\}_{\ell \in [k]}$ ($k \le n$) on $\mathcal{S}_{n,\epsilon}$ is a collection of strands on $\mathcal{S}_{n,\epsilon}$ that satisfies the following conditions:

$\begin{array}{rll}
a) & \text{distinct strands do not intersect nontrivially, and} \\
b) & \text{the graph determined by $d$ is a \textbf{forest} (i.e. a disjoint union of trees)} \\
\end{array}$

\noindent Let $\mathcal{D}_{k,\epsilon}$ denote the set of strand diagrams on $\mathcal{S}_{n,\epsilon}$ with $k$ strands and let $\mathcal{D}_\epsilon$ denote the set of strand diagrams with any positive number of strands. Then $$\mathcal{D}_\epsilon = \bigsqcup_{k \in [n]} \mathcal{D}_{k,\epsilon}.$$
\end{definition}

\begin{example}\label{firstexample}
Let $n = 4$ and $\epsilon = (-,+, -, +,+)$ so that $Q_{\epsilon} = 1 \stackrel{a_1^+}{\longrightarrow} 2 \stackrel{a_2^-}{\longleftarrow} 3 \stackrel{a_3^+}{\longrightarrow} 4.$ Then we have that $d_1 = \{c(0,1), c(0,2), c(2,3), c(2,4)\} \in \mathcal{D}_{4,\epsilon}$ and $d_2 = \{c(0,4), c(1,3), c(2,4)\}\in \mathcal{D}_{3,\epsilon}$. We draw these strand diagrams below.
$$\begin{array}{ccccccc}
\includegraphics[scale=1]{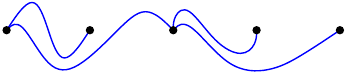} & & & & \includegraphics[scale=1]{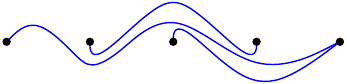}
\end{array}$$
\end{example}

The following technical lemma classifies when two distinct indecomposable representations of $Q_\epsilon$ define 0, 1, or 2 exceptional pairs. Its proof appears in Section~\ref{firstproofs}.

\begin{lemma}\label{maintechlemma}
Let $Q_\epsilon$ be given. Fix two distinct indecomposable representations $U, V \in \text{ind}(\text{rep}_\Bbbk(Q_\epsilon))$.

$\begin{array}{rll}
a) & \text{The strands $\Phi_{\epsilon}(U)$ and $\Phi_{\epsilon}(V)$ intersect nontrivially if and only if neither $(U,V)$ nor $(V,U)$ are}\\
& \text{exceptional pairs.}\\
b) & \text{The strand $\Phi_{\epsilon}(U)$ is clockwise from $\Phi_{\epsilon}(V)$  if and only if $(U,V)$ is an exceptional pair and $(V,U)$}\\
& \text{is not an exceptional pair.}\\
c) & \text{The strands $\Phi_{\epsilon}(U)$ and $\Phi_{\epsilon}(V)$ do not intersect at any of their endpoints and they do not intersect}\\
& \text{nontrivially if and only if $(U, V)$ and $(V, U)$ are both exceptional pairs.}
\end{array}$
\end{lemma}

Using Lemma~\ref{maintechlemma} we obtain our first main result. The following theorem says that the data of an exceptional collection is completely encoded in the strand diagram it defines.

\begin{theorem}\label{ECbij}
Let $\overline{\mathcal{E}}_\epsilon := \{\text{exceptional collections of } Q_\epsilon\}$. There is a bijection $\overline{\mathcal{E}}_\epsilon \to \mathcal{D}_{\epsilon}$ defined by $$\overline{\xi}_\epsilon = \{X^\epsilon_{i_\ell,j_\ell}\}_{\ell \in [k]} \mapsto \{\Phi_{\epsilon}(X^\epsilon_{i_\ell,j_\ell})\}_{\ell \in [k]}.$$
\end{theorem}

\begin{proof}
Let $\overline{\xi}_\epsilon = \{X^\epsilon_{i_\ell,j_\ell}\}_{\ell \in [k]}$ be an exceptional collection of $Q_\epsilon$.  Let $\xi_\epsilon$ be an exceptional sequence gotten from $\overline{\xi}_\epsilon$ by reordering its representations.  Without loss of generality, assume $\xi_\epsilon = (X^\epsilon_{i_\ell,j_\ell})_{\ell \in [k]}$ is an exceptional sequence.  Thus, $(X^\epsilon_{i_\ell,j_\ell},X^\epsilon_{i_p,j_p})$ is an exceptional pair for all $\ell < p$.  Lemma~\ref{maintechlemma} a) implies that distinct strands of $\{\Phi_{\epsilon}(X^\epsilon_{i_\ell,j_\ell})\}_{\ell \in [k]}$ do not intersect nontrivially.

Now we will show that $\{\Phi_{\epsilon}(X^\epsilon_{i_\ell,j_\ell})\}_{\ell \in [k]}$ has no cycles.  Suppose that $\Phi_{\epsilon}(X^\epsilon_{i_{\ell_1},j_{\ell_1}}),\ldots,\Phi_{\epsilon}(X^\epsilon_{i_{\ell_p},j_{\ell_p}})$ is a cycle of length $p \leq k$ in $\Phi_{\epsilon}(\xi_\epsilon)$. Then, there exist $\ell_a, \ell_b \in [k]$ with $\ell_b > \ell_a$ such that $\Phi_{\epsilon}(X^\epsilon_{i_{\ell_b},j_{\ell_b}})$ is clockwise from $\Phi_{\epsilon}(X^\epsilon_{i_{\ell_a},j_{\ell_a}})$. Thus, by Lemma~\ref{maintechlemma} b), $(X^\epsilon_{i_{\ell_a},j_{\ell_a}},X^\epsilon_{i_{\ell_b},j_{\ell_b}})$ is not an exceptional pair.  This contradicts the fact that $(X^\epsilon_{i_{\ell_1},j_{\ell_1}},\ldots,X^\epsilon_{i_{\ell_p},j_{\ell_p}})$ is an exceptional sequence.  Hence, the graph determined by $\{\Phi_{\epsilon}(X^\epsilon_{i_\ell,j_\ell})\}_{\ell \in [k]}$ is a tree.  We have shown that $\Phi_{\epsilon}(\overline{\xi}_\epsilon) \in \mathcal{D}_{k,\epsilon}$.

Now let $d = \{c(i_\ell, j_\ell)\}_{\ell \in [k]} \in \mathcal{D}_{k,\epsilon}$.  Since $c(i_\ell,j_\ell)$ and $c(i_m,j_m)$ do not intersect nontrivially, it follows that $(\Phi_{\epsilon}^{-1}(c(i_\ell,j_\ell)),\Phi_{\epsilon}^{-1}(c(i_m,j_m)))$ or $(\Phi_{\epsilon}^{-1}(c(i_m,j_m)),\Phi_{\epsilon}^{-1}(c(i_\ell,j_\ell)))$ is an exceptional pair for every $\ell \neq m$.  Notice that there exists $c(i_{\ell_1},j_{\ell_1}) \in d$ such that $(\Phi_{\epsilon}^{-1}(c(i_{\ell_1},j_{\ell_1})),\Phi_{\epsilon}^{-1}(c(i_\ell,j_\ell)))$ is an exceptional pair for every $c(i_\ell,j_\ell) \in d \setminus \{c(i_{\ell_1},j_{\ell_1})\}$.  This is true because if such $c(i_{\ell_1},j_{\ell_1})$ did not exist, then $d$ must have a cycle.  Set $E_1 = \Phi_{\epsilon}^{-1}(c(i_{\ell_1},j_{\ell_1}))$.  Now, choose $c(i_{\ell_p},j_{\ell_p})$ such that $(\Phi_{\epsilon}^{-1}(c(i_{\ell_p},j_{\ell_p})),\Phi_{\epsilon}^{-1}(c(i_\ell,j_\ell)))$ is an exceptional pair for every $c(i_\ell,j_\ell) \in d \setminus \{c(i_{\ell_1},j_{\ell_1}),\ldots,c(i_{\ell_p},j_{\ell_p})\}$ inductively and put $E_p = \Phi_{\epsilon}^{-1}(c(i_{\ell_p},j_{\ell_p}))$.  By construction, $(E_1,\ldots,E_k)$ is a complete exceptional sequence, as desired.
\end{proof}

Our next step is to add distinct integer labels to each strand in a given strand diagram $d$. When these labels have what we call a \textbf{good} labeling, these labels will describe exactly the order in which to put the representations corresponding to strands of $d$ so that the resulting sequence of representations is an exceptional sequence.

\begin{definition}\label{Def:labeled_diag}
A \textbf{labeled diagram} $d(k) = \{(c(i_\ell, j_\ell), s_\ell)\}_{\ell \in [k]}$ on $\mathcal{S}_{n,\epsilon}$ is a set of pairs $(c(i_\ell, j_\ell), s_\ell)$ where $c(i_\ell, j_\ell)$ is a strand on $\mathcal{S}_{n,\epsilon}$ and $s_\ell \in [k]$ such that $d := \{c(i_\ell, j_\ell)\}_{\ell \in [k]}$ is a strand diagram on $\mathcal{S}_{n,\epsilon}$ and $s_\ell \neq s_{\ell^\prime}$ for any distinct $\ell, \ell^\prime \in [k]$. We refer to the pairs $(c(i_\ell, j_\ell), s_\ell)$ as \textbf{labeled strands} and to $d$ as the \textbf{underlying diagram} of $d(k)$. We define the \textbf{endpoints} of a labeled strand $(c(i_\ell, j_\ell), s_\ell)$ to be the endpoints of $c(i_\ell, j_\ell)$. 

Let $\epsilon_i \in \mathcal{S}_{n,\epsilon}$ and let $((c(i,j_1),s_1),\ldots, (c(i,j_r),s_r))$ be the sequence of all labeled strands of $d(k)$ that have $\epsilon_i$ as an endpoint, and assume that this sequence is ordered so that strand $c(i,j_k)$ is clockwise from $c(i,j_{k^\prime})$ if $k^\prime < k$. We say the strand labeling of $d(k)$ is \textbf{good} if for each point $\epsilon_i \in \mathcal{S}_{n,\epsilon}$ one has $s_1 < \cdots < s_r.$ Let $\mathcal{D}_{k,\epsilon}(k)$ denote the set of labeled strand diagrams on $\mathcal{S}_{n,\epsilon}$ with $k$ strands and with good strand labelings.
\end{definition}

\begin{example}\label{secondexample} 
Let $n = 4$ and $\epsilon = (-,+, -, +,+)$ so that $Q_{\epsilon} = 1 \stackrel{a_1^+}{\longrightarrow} 2 \stackrel{a_2^-}{\longleftarrow} 3 \stackrel{a_3^+}{\longrightarrow} 4.$ Below we show the labeled diagrams $d_1(4) = \{(c(0,1),1), (c(0,2), 2), (c(2,3), 3), (c(2,4), 4)\}$ and $d_2(3) = \{(c(0,4), 1), (c(2,4), 2), (c(1,3), 3)\}.$

$$\begin{array}{ccccccc}
\includegraphics[scale=1]{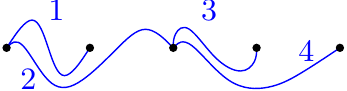} & & & & \includegraphics[scale=1]{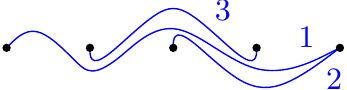}
\end{array}$$

\noindent We have that $d_1(4) \in \mathcal{D}_{4,\epsilon}(4)$, but $d_2(3) \not \in \mathcal{D}_{3,\epsilon}(3).$
\end{example}

\begin{theorem}\label{ESbij}
Let $k \in [n]$ and let $\mathcal{E}_\epsilon(k) := \{\text{exceptional sequences of } Q_\epsilon\text{ of length }k\}.$ There is a bijection $\widetilde{\Phi}_\epsilon: \mathcal{E}_\epsilon(k) \to \mathcal{D}_{k,\epsilon}(k)$ defined by 
$$\xi_\epsilon = (X^\epsilon_{i_\ell,j_\ell})_{\ell \in [k]} \longmapsto \{(c(i_\ell,j_\ell), k+1-\ell)\}_{\ell \in [k]}.$$
\end{theorem}

\begin{proof}
Let $\xi_\epsilon \in \mathcal{E}_\epsilon(k)$. By Lemma~\ref{maintechlemma} a), $\widetilde{\Phi}_{\epsilon}(\xi_\epsilon)$ has no strands that intersect nontrivially.  Let $(V_1,V_2)$ be an exceptional pair appearing in $\xi_\epsilon$ with $V_i$ corresponding to strand $c_i$ in $\widetilde{\Phi}_{\epsilon}(\xi_\epsilon)$ for $i=1,2$, where $c_1$ and $c_2$ intersect only at one of their endpoints.  Note that by the definition of $\widetilde{\Phi}_{\epsilon}$, the strand label of $c_1$ is larger than that of $c_2$.  From Lemma~\ref{maintechlemma} b), strand $c_1$ is clockwise from $c_2$ in $\widetilde{\Phi}_{\epsilon}(\xi_\epsilon)$.  Thus the strand labeling of $\widetilde{\Phi}_{\epsilon}(\xi_\epsilon)$ is good, so $\widetilde{\Phi}_{\epsilon}(\xi_\epsilon) \in \cD_{k,\epsilon}(k)$ for any $\xi_\epsilon \in \mathcal{E}_\epsilon(k)$.

Let $\widetilde{\Psi}_{\epsilon}: \cD_{k,\epsilon}(k) \rightarrow \mathcal{E}_\epsilon(k)$ be defined by $\{(c(i_\ell,j_\ell),\ell)\}_{\ell \in[k]} \mapsto (X_{i_k,j_k}^{\epsilon},X_{i_{k-1},j_{k-1}}^\epsilon,\ldots,X_{i_1,j_1}^\epsilon)$.  We will show that $\widetilde{\Psi}_\epsilon(d(k)) \in \mathcal{E}_\epsilon(k)$ for any $d(k) \in \mathcal{D}_{k,\epsilon}(k)$ and that $\widetilde{\Psi}_\epsilon = \widetilde{\Phi}_{\epsilon}^{-1}.$ Let $$\widetilde{\Psi}_\epsilon(\{(c(i_\ell, j_\ell), \ell)\}_{\ell \in [k]}) = (X_{i_k,j_k}^\epsilon, X_{i_{k-1}, j_{k-1}}^\epsilon,\ldots, X_{i_1, j_1}^\epsilon).$$ Consider the pair $(X_{i_s,j_s}^\epsilon, X_{i_{s^\prime}, j_{s^\prime}}^\epsilon)$ with $s > s^\prime.$ We will show that $(X_{i_s,j_s}^\epsilon, X_{i_{s^\prime}, j_{s^\prime}}^\epsilon)$ is an exceptional pair and thus conclude that $\widetilde{\Psi}_\epsilon(\{(c(i_\ell, j_\ell), \ell)\}_{\ell \in [k]}) \in \mathcal{E}_\epsilon(k)$ for any $d(k) \in \mathcal{D}_{k,\epsilon}(k)$. Clearly, $c(i_s,j_s)$ and $c(i_{s^\prime}, j_{s^\prime})$ do not intersect nontrivially. If $c(i_s, j_s)$ and $c(i_{s^\prime}, j_{s^\prime})$ do not intersect at one of their endpoints, then by Lemma~\ref{maintechlemma} c) $(X_{i_s,j_s}^\epsilon, X_{i_{s^\prime}, j_{s^\prime}}^\epsilon)$ is exceptional. Now suppose $c(i_s, j_s)$ and $c(i_{s^\prime}, j_{s^\prime})$ intersect at one of their endpoints. Because the strand-labeling of $\{(c(i_\ell, j_\ell), \ell)\}_{\ell \in [k]}$ is good, $c(i_s, j_s)$ is clockwise from  $c(i_{s^\prime}, j_{s^\prime})$. By Lemma~\ref{maintechlemma} b), we have that $(X_{i_s,j_s}^\epsilon, X_{i_{s^\prime}, j_{s^\prime}}^\epsilon)$ is exceptional.

To see that $\widetilde{\Psi}_\epsilon = \widetilde{\Phi}_\epsilon^{-1}$, observe that
$$\begin{array}{rcl}
\widetilde{\Phi}_\epsilon \left(\widetilde{\Psi}_\epsilon(\{(c(i_\ell, j_\ell), \ell)\}_{\ell \in [k]})\right) & = & \widetilde{\Phi}_\epsilon \left((X_{i_k,j_k}^\epsilon, X_{i_{k-1}, j_{k-1}}^\epsilon,\ldots, X_{i_1,j_1}^\epsilon)\right)\\
& = & \{(c(i_\ell, j_\ell), k+1 - (k+1 - \ell))\}_{\ell \in [k]} \\
& = & \{(c(i_\ell, j_\ell), \ell)\}_{\ell \in [k]}.
\end{array}$$
Thus $\widetilde{\Phi}_\epsilon \circ \widetilde{\Psi}_\epsilon = 1_{\mathcal{D}_{n,\epsilon}(k)}.$ Similarly, one shows that $\widetilde{\Psi}_\epsilon \circ \widetilde{\Phi}_\epsilon = 1_{\mathcal{E}_\epsilon(k)}.$ Thus $\widetilde{\Phi}_\epsilon$ is a bijection.
\end{proof}

The last combinatorial objects we discuss in this section are called \textbf{oriented diagrams}. These are strand diagrams whose strands have a direction. 

\begin{definition}
Let $\overrightarrow{c}(i_\ell, j_\ell)$ denote the data of the strand $c(i_\ell,j_\ell)$ on $\mathcal{S}_{n,\epsilon}$ and an orientation of each curve in $c(i_\ell, j_\ell)$ from $\epsilon_{i_\ell}$ to $\epsilon_{j_\ell}$. We refer to $\overrightarrow{c}(i_\ell, j_\ell)$ as an \textbf{oriented strand} on $\mathcal{S}_{n,\epsilon}$ and we define the \textbf{endpoints} of $\overrightarrow{c}(i_\ell, j_\ell)$ to be the endpoints of $c(i_\ell, j_\ell)$. An \textbf{oriented diagram} $\overrightarrow{d}=\{\overrightarrow{c}(i_\ell, j_\ell)\}_{\ell \in [k]}$ on $\mathcal{S}_{n,\epsilon}$ is a collection of oriented strands on $\mathcal{S}_{n,\epsilon}$ where $d = \{c(i_\ell, j_\ell)\}_{\ell \in [k]}$ is a strand diagram on $\mathcal{S}_{n,\epsilon}.$ We refer to $d$ as the \textbf{underlying diagram} of $\overrightarrow{d}$. 
\end{definition}

\begin{remark}
When it is clear from the context what the values of $n$ and $\epsilon$ are, we will often refer to a strand diagram on $\mathcal{S}_{n,\epsilon}$ simply as a \textbf{diagram}. Similarly, we will often refer to labeled diagrams (resp. oriented diagrams) on $\mathcal{S}_{n,\epsilon}$ as \textbf{labeled diagrams} (resp. \textbf{oriented diagrams}). Additionally, if we have two diagrams $d_1$ and $d_2$ (both on $\mathcal{S}_{n, \epsilon}$) where $d_1 \subset d_2$, we say that $d_1$ is a \textbf{subdiagram} of $d_2$. One analogously defines \textbf{labeled subdiagrams} (resp. \textbf{oriented subdiagrams}) of a labeled diagram (resp. oriented diagram).\end{remark}

We now define a special subset of the oriented diagrams on $\mathcal{S}_{n,\epsilon}$. As we will see, each element in this subset of oriented diagrams, denoted $\overrightarrow{\mathcal{D}}_{n,\epsilon}$, will correspond to a unique $\textbf{c}$-matrix $C \in \textbf{c}\text{-mat}(Q_\epsilon)$ and vice versa. Thus we obtain a diagrammatic classification of \textbf{c}-matrices (see Theorem~\ref{c-matClassif}).

\begin{definition}\label{def:cmatdiag}
Let $\overrightarrow{\mathcal{D}}_{n,\epsilon}$ denote the set of oriented diagrams $\overrightarrow{d} = \{\overrightarrow{c}(i_\ell,j_\ell)\}_{\ell \in [n]}$ on $\mathcal{S}_{n,\epsilon}$ with the property that for each $k \in [0,n]$ there exist integers $i_1, i_2, k, j \in [0,n]$ where $i_1< k< i_2$ and $j \in [0,n]\backslash\{i_1, k, i_2\}$ such that the oriented subdiagram $\overrightarrow{d}_1$ of $\overrightarrow{d}$ consisting of the oriented strands of $\overrightarrow{d}$ with $\epsilon_k$ as an endpoint is an oriented subdiagram of one of the following two oriented diagrams on $\mathcal{S}_{n,\epsilon}$:

\begin{itemize}
\item[i)] $\overrightarrow{d}_+ = \{\overrightarrow{c}(k,i_1),\overrightarrow{c}(k,i_2), \overrightarrow{c}(j,k)\}$ where $\epsilon_k = +$ (shown in Figure~\ref{local1} (left)) or
\item[ii)] $\overrightarrow{d}_- = \{\overrightarrow{c}(i_1,k),\overrightarrow{c}(i_2,k), \overrightarrow{c}(k,j)\}$ where $\epsilon_k = -$ (shown in Figure~\ref{local1} (right)).
\end{itemize}
\end{definition}

\begin{figure}[h]
\includegraphics[scale=1.5]{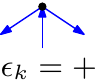} \ \ \ \ \ \ \ \ \ \  \includegraphics[scale=1.5]{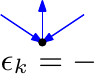}
\caption{}
\label{local1}
\end{figure}

\begin{lemma}\label{cmatdiag}
Let $\{\overrightarrow{c_i}\}_{i \in [k]}$ be a collection of \textbf{c}-vectors of $Q_\epsilon$ where $k \le n$. Let $\overrightarrow{c_i} = \pm\underline{\dim}(X^\epsilon_{i_1,i_2})$ where the sign is determined by $\overrightarrow{c_i}$. There is an injective map 
$$\begin{array}{rcl}
\{ \text{\textbf{noncrossing collections} $\{\overrightarrow{c_i}\}_{i\in[k]}$ of $Q_\epsilon$}\} & \longrightarrow & \{\text{oriented diagrams } \overrightarrow{d} = \{\overrightarrow{c}(i_\ell, j_\ell)\}_{\ell \in [k]}\}\end{array}$$
defined by
$$\begin{array}{rcl}
\overrightarrow{c_i} &\longmapsto & \left\{\begin{array}{r c l c c c c c} \overrightarrow{c}(i_1,i_2) & : & \text{\overrightarrow{c_i} is positive}\\ \overrightarrow{c}(i_2,i_1) & : & \text{\overrightarrow{c_i} is negative,}
\end{array}\right.
\end{array}$$

\noindent where $\{\overrightarrow{c_i}\}_{i \in [k]}$ is a \textbf{noncrossing collection} of \textbf{c}-vectors if $\Phi_{\epsilon}(X^\epsilon_{i_1,i_2})$ and $\Phi_{\epsilon}(X^\epsilon_{i^\prime_1,i^\prime_2})$ do not intersect nontrivially for any $i,i^\prime \in [k]$. In particular, each \textbf{c}-matrix $C_\epsilon \in \textbf{c}\text{-mat}(Q_\epsilon)$ determines a unique oriented diagram denoted $\overrightarrow{d}_{C_\epsilon}$ with $n$ oriented strands.
\end{lemma}

\begin{example}\label{thirdexample}
Let $n = 4$ and $\epsilon = (+,+, -, +,-)$ so that $Q_{\epsilon}= 1 \stackrel{a_1^+}{\longrightarrow} 2 \stackrel{a_2^-}{\longleftarrow} 3 \stackrel{a_3^+}{\longrightarrow} 4.$  After performing the mutation sequence $\mu_3 \circ \mu_2$ to the corresponding framed quiver, we have the $\bc$-matrix with its oriented diagram.
$$
\left[
\begin{array}{c c c c}
1 & 1 & 0 & 0 \\
0 & 0 & 1 & 0 \\
0 & -1 & -1 & 0 \\
0 & 0 & 0 & 1
\end{array}\right] \ \ \ \ \ \ \ \ \ \  
\raisebox{-.1in}{\includegraphics[scale=1.5]{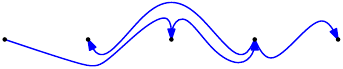}}
$$
\end{example}

The following theorem shows oriented diagrams belonging to $\overrightarrow{\mathcal{D}}_{n,\epsilon}$ are in bijection with \textbf{c}-matrices of $Q_\epsilon.$  We delay its proof until Section~\ref{sec:mct} because it makes use of the concept of a mixed cobinary tree.

\begin{theorem}\label{c-matClassif}
The map $\textbf{c}\text{-mat}(Q_\epsilon)\to \overrightarrow{\mathcal{D}}_{n,\epsilon}$ induced by the map defined in Lemma \ref{cmatdiag} is a bijection.
\end{theorem}

\subsection{Proof of Lemma~\ref{maintechlemma}}\label{firstproofs}

The proof of Lemma~\ref{maintechlemma} requires some notions from representation theory of finite dimensional algebras, which we now briefly review. For a more comprehensive treatment of the following notions, we refer the reader to \cite{ass06}.

\begin{definition}
Given a quiver $Q$ with $\#Q_0=n$, the \textbf{Euler characteristic} (of $Q$) is the $\ZZ$-bilinear (nonsymmetric) form $\ZZ^n \times \ZZ^n \rightarrow \ZZ$ defined by
\[
\langle \underline{\dim}(V),\underline{\dim}(W)\rangle = \sum_{i\ge 0}(-1)^i\dim \Ext_{\Bbbk Q}^i(V,W)
\]
for every $V,W \in \rep_\Bbbk(Q)$.
\end{definition}

For hereditary algebras $A$ (e.g. path algebras of acyclic quivers), $\Ext_{A}^i(V,W)=0$ for $i \geq 2$ and the formula reduces to

$$\langle \underline{\dim}(V),\underline{\dim}(W)\rangle = \dim \Hom_{\Bbbk Q}(V,W) - \dim \Ext_{\Bbbk Q}^1(V,W)$$

The following result gives a simple formula for the Euler characteristic. We note that this formula is independent of the orientation of the arrows of $Q$. The following lemma can be found in \cite{ass06}. 

\begin{lemma} Given an acyclic quiver $Q$ with $\#Q_0=n$ and integral vectors $x=(x_1,x_2,\ldots,x_n),y=(y_1,y_2,\ldots,y_n) \in \ZZ^n$, the Euler characteristic of $Q$ has the form
\[
\langle x,y\rangle = \sum_{i \in Q_0}x_iy_i - \sum_{\alpha \in Q_1}x_{s(\alpha)}y_{t(\alpha)}
\]
\end{lemma}

Next, we give a slight simplification of the previous formula. Recall that the support of $V \in \rep_\Bbbk(Q)$ is the set $\text{supp}(V) := \{i\in Q_0 : V_i \neq 0\}$. Thus for quivers of the form $Q_\epsilon$, any indecomposable representation $X^\epsilon_{i,j} \in \text{ind}(\text{rep}_\Bbbk(Q_\epsilon))$ has $\text{supp}(X^\epsilon_{i,j}) = [i+1,j].$

\begin{lemma}\label{eulerform}
Let $X^\epsilon_{k,\ell}, X^\epsilon_{i,j} \in \text{ind}(\text{rep}_\Bbbk(Q_\epsilon))$ and $A := \{a \in (Q_\epsilon)_1: s(a), t(a) \in \text{supp}(X^\epsilon_{k,\ell})\cap \text{supp}(X^\epsilon_{i,j})\}$. Then 
$$\langle \underline{\dim}(X^\epsilon_{k,\ell}), \underline{\dim}(X^\epsilon_{i,j})\rangle = \chi_{\text{supp}(X^\epsilon_{k,\ell}) \cap \text{supp}(X^\epsilon_{i,j})} - \#\left(\{a \in (Q_\epsilon)_1: s(a) \in \text{supp}(X^\epsilon_{k,\ell}), \ t(a) \in \text{supp}(X^\epsilon_{i,j})\}\backslash A\right)$$
where $\chi_{\text{supp}(X^\epsilon_{k,\ell}) \cap \text{supp}(X^\epsilon_{i,j})} = 1$ if $\text{supp}(X^\epsilon_{k,\ell})\cap \text{supp}(X^\epsilon_{i,j}) \neq \emptyset$ and 0 otherwise.\end{lemma}

\begin{proof} We have that
$$\begin{array}{rcl}
 \langle \underline{\dim}(X^\epsilon_{k,\ell}), \underline{\dim}(X^\epsilon_{i,j})\rangle & = & \displaystyle \sum_{m \in ({Q_\epsilon})_0} \underline{\dim}(X^\epsilon_{k,\ell})_m\underline{\dim}(X^\epsilon_{i,j})_m
- \sum_{a \in ({Q_\epsilon})_1} \underline{\dim}(X^\epsilon_{k,\ell})_{s(a)}\underline{\dim}(X^\epsilon_{i,j})_{t(a)} \\
& = &  \#\left(\text{supp}(X^\epsilon_{k,\ell})\cap \text{supp}(X^\epsilon_{i,j})\right)\\
& &- \#\{\alpha\in (Q_\epsilon)_1: s(a) \in \text{supp}(X^\epsilon_{k,\ell}), \ t(a) \in \text{supp}(X^\epsilon_{i,j})\}\\
& = &  \#\left(\text{supp}(X^\epsilon_{k,\ell})\cap \text{supp}(X^\epsilon_{i,j})\right) - \#A\\
& &- \#\left(\{a \in (Q_\epsilon)_1: s(a) \in \text{supp}(X^\epsilon_{k,\ell}), \ t(a) \in \text{supp}(X^\epsilon_{i,j})\}\backslash A\right).
\end{array}$$
Observe that if $\text{supp}(X^\epsilon_{k,\ell})\cap \text{supp}(X^\epsilon_{i,j}) \neq \emptyset$, then $\#A = \#(\text{supp}(X^\epsilon_{k,\ell})\cap \text{supp}(X^\epsilon_{i,j})) - 1.$ Otherwise $\#A = 0.$ Thus $$\langle \underline{\dim}(X^\epsilon_{k,\ell}), \underline{\dim}(X^\epsilon_{i,j})\rangle = \chi_{\text{supp}(X^\epsilon_{k,\ell}) \cap \text{supp}(X^\epsilon_{i,j})} - \#\left(\{a \in (Q_\epsilon)_1: s(a) \in \text{supp}(X^\epsilon_{k,\ell}), \ t(a) \in \text{supp}(X^\epsilon_{i,j})\}\backslash A\right)  .$$\end{proof}

We now present the lemmas that we will use in the proof of Lemma~\ref{maintechlemma}. The proofs of the next four lemmas use very similar techniques so we only prove Lemma~\ref{interlaced}. The following four lemmas characterize when $\Hom_{\Bbbk Q_\epsilon}(-,-)$ and $\Ext^1_{\Bbbk Q_\epsilon}(-,-)$ vanish for a given type $\mathbb{A}_n$ quiver $Q_\epsilon.$ The conditions describing when $\Hom_{\Bbbk Q_\epsilon}(-,-)$ and $\Ext^1_{\Bbbk Q_\epsilon}(-,-)$ vanish are given in terms of inequalities satisfied by the indices that describe a pair of indecomposable representations of $Q_\epsilon$ and the entries of $\epsilon.$

\begin{lemma}\label{interlaced}
Let $X^\epsilon_{k,\ell}, X^\epsilon_{i,j} \in \text{ind}(\text{rep}_\Bbbk(Q_\epsilon))$. Assume $0 \le i < k < j < \ell \le n$.

$\begin{array}{rll}
i) & \text{$\Hom_{\Bbbk Q_\epsilon}(X^\epsilon_{i,j}, X^\epsilon_{k,\ell}) \neq 0$ if and only if $\epsilon_k = -$ and $\epsilon_j = -$.}\\
ii) & \text{$\Hom_{\Bbbk Q_\epsilon}(X^\epsilon_{k,\ell}, X^\epsilon_{i,j}) \neq 0$ if and only if $\epsilon_k = +$ and $\epsilon_j = +$.}\\
iii) & \text{$\Ext^1_{\Bbbk Q_\epsilon}(X^\epsilon_{i,j}, X^\epsilon_{k,\ell}) \neq 0$ if and only if $\epsilon_k = +$ and $\epsilon_j = +$.}\\
iv) & \text{$\Ext^1_{\Bbbk Q_\epsilon}(X^\epsilon_{k,\ell}, X^\epsilon_{i,j}) \neq 0$ if and only if $\epsilon_k = -$ and $\epsilon_j = -$.}
\end{array}$
\end{lemma}
\begin{proof}
We only prove $i)$ and $iv)$ as the proof of $ii)$ is very similar to that of $i)$, and the proof of $iii)$ is very similar to that of $iv)$. To prove $i)$, first assume there is a nonzero morphism $\theta: X^\epsilon_{i,j} \to X^\epsilon_{k,\ell}.$ Clearly, $\theta_s = 0$ if $s \not \in [k+1,j]$. If $\theta_s \neq 0$ for some $s \in [n]$, then $\theta_s = \lambda\cdot \text{id}_{\Bbbk}$ for some nonzero $\lambda \in \Bbbk$ (i.e. $\theta_s$ is a nonzero scalar transformation). As $\theta$ is a morphism of representations, for any $a \in (Q_\epsilon)_1$ the equality $\theta_{t(a)}\varphi^{i,j}_a = \varphi^{k,\ell}_a \theta_{s(a)}$ holds. Thus for any $a \in \{a_{k+1}^{\epsilon_{k+1}}, \ldots, a_{j - 1}^{\epsilon_{j - 1}}\}$, we have $\theta_{t(a)} = \theta_{s(a)}.$ As $\theta$ is nonzero, this implies that $\theta_s = \lambda\cdot \text{id}_{\Bbbk}$ for any $s \in [k+1,j]$. If $a = a_{k}^{\epsilon_k}$, then we have 
$$\begin{array}{rcl}
\theta_{t(a)}\varphi^{i,j}_a & = & \varphi^{k,\ell}_a\theta_{s(a)} \\
\theta_{t(a)} & = & 0.
\end{array}$$
Thus $\epsilon_k = -$. Similarly, $\epsilon_j = -.$ 

Conversely, it is easy to see that if $\epsilon_k = \epsilon_j = -,$ then $\theta: X^\epsilon_{i,j} \to X^\epsilon_{k,\ell}$ defined by $\theta_s = 0$ if $s \not \in [k+1,j]$ and $\theta_s = 1$ otherwise is a nonzero morphism.

Next, we prove $iv)$. Observe that by Lemma~\ref{eulerform} we have
$$\begin{array}{rcl}
\dim\text{Ext}^1_{\Bbbk Q_\epsilon}(X^\epsilon_{k,\ell}, X^{\epsilon}_{i,j}) & = & \dim\text{Hom}_{\Bbbk Q_\epsilon}(X^\epsilon_{k,\ell},X^\epsilon_{i,j}) -\langle \underline{\dim}(X^\epsilon_{k,\ell}), \underline{\dim}(X^\epsilon_{i,j})\rangle \\
& = & \dim\text{Hom}_{\Bbbk Q_\epsilon}(X^\epsilon_{k,\ell},X^\epsilon_{i,j}) - 1\\
& & + \#\left(\{b \in (Q_\epsilon)_1: s(b) \in \text{supp}(X^\epsilon_{k,\ell}), \ t(b) \in \text{supp}(X^\epsilon_{i,j})\}\backslash A\right). 
\end{array}$$
Note that $\#\left(\{b \in (Q_\epsilon)_1: s(b) \in \text{supp}(X^\epsilon_{k,\ell}), \ t(b) \in \text{supp}(X^\epsilon_{i,j})\}\backslash A\right) \le 2$ with equality if and only if $\epsilon_k = \epsilon_j = -.$ 

Suppose $\epsilon_k = \epsilon_j = -$. Then by $i)$, we have that $\text{Hom}_{\Bbbk Q_\epsilon}(X_{i,j}^\epsilon, X_{k,\ell}^\epsilon) \neq 0$ so $\text{Hom}_{\Bbbk Q_\epsilon}(X_{k,\ell}^\epsilon, X_{i,j}^\epsilon) = 0$. This means 
$$\begin{array}{rcl}
\dim\text{Ext}^1_{\Bbbk Q_\epsilon}(X_{k,\ell}^\epsilon,X_{i,j}^\epsilon) & = & \#\left(\{b \in (Q_\epsilon)_1: s(b) \in \text{supp}(X^\epsilon_{k,\ell}), \ t(b) \in \text{supp}(X^\epsilon_{i,j})\}\backslash A\right) - 1\\
& = & 1.
\end{array}$$

Conversely, suppose $\text{Ext}^1_{\Bbbk Q_\epsilon}(X_{k,\ell}^\epsilon, X_{i,j}^\epsilon) \neq 0$. Thus, one checks that there is a nonsplit extension $$0 \longrightarrow X^\epsilon_{i,j} \stackrel{f}{\longrightarrow} X^\epsilon_{i,\ell} \oplus X^\epsilon_{k,j} \stackrel{g}{\longrightarrow} X^\epsilon_{k,\ell} \longrightarrow 0.$$ This implies that $\text{Hom}_{\Bbbk Q_\epsilon}(X_{k,\ell}^\epsilon, X_{i,j}^\epsilon) = 0$, since the composition $h: X^\epsilon_{i,j} \stackrel{f_1}{\to} X^\epsilon_{i,\ell} \stackrel{g_1}{\to} X_{k,\ell}^\epsilon$ is nonzero. Using again that $\text{dim}\text{Ext}^1_{\Bbbk Q_\epsilon}(X_{k,\ell}^\epsilon, X_{i,j}^\epsilon) \neq 0$, the formula above for this dimension tells us that $\epsilon_k = \epsilon_j = -$.

\end{proof}

\begin{lemma}\label{nested}
Let $X^\epsilon_{k,\ell}, X^\epsilon_{i,j} \in \text{ind}(\text{rep}_\Bbbk(Q_\epsilon))$. Assume $0 \le i < k < \ell < j \le n$.

$\begin{array}{rll}
i) & \text{$\Hom_{\Bbbk Q_\epsilon}(X^\epsilon_{i,j}, X^\epsilon_{k,\ell}) \neq 0$ if and only if $\epsilon_k = -$ and $\epsilon_\ell = +$.}\\
ii) & \text{$\Hom_{\Bbbk Q_\epsilon}(X^\epsilon_{k,\ell}, X^\epsilon_{i,j}) \neq 0$ if and only if $\epsilon_k = +$ and $\epsilon_\ell = -$.}\\
iii) & \text{$\Ext^1_{\Bbbk Q_\epsilon}(X^\epsilon_{i,j}, X^\epsilon_{k,\ell}) \neq 0$ if and only if $\epsilon_k = +$ and $\epsilon_\ell = -$.}\\
iv) & \text{$\Ext^1_{\Bbbk Q_\epsilon}(X^\epsilon_{k,\ell}, X^\epsilon_{i,j}) \neq 0$ if and only if $\epsilon_k = -$ and $\epsilon_\ell = +$.}
\end{array}$
\end{lemma}

\begin{lemma}\label{shareendpt}
Assume $0 \le i < k < j \le n$. Then

$\begin{array}{rll}
i) & \text{$\Hom_{\Bbbk Q_\epsilon}(X^\epsilon_{i,k}, X^\epsilon_{k,j}) = 0$ and $\Hom_{\Bbbk Q_\epsilon}(X^\epsilon_{k,j}, X^\epsilon_{i,k}) = 0$.}\\
ii) & \text{$\Ext^1_{\Bbbk Q_\epsilon}(X^\epsilon_{i,k}, X^\epsilon_{k,j}) \neq 0$ if and only if $\epsilon_k = +$.}\\
iii) & \text{$\Ext^1_{\Bbbk Q_\epsilon}(X^\epsilon_{k,j}, X^\epsilon_{i,k}) \neq 0$ if and only if $\epsilon_k = -$.}\\
iv) & \text{$\Hom_{\Bbbk Q_\epsilon}(X^\epsilon_{i,k}, X^\epsilon_{i,j}) \neq 0$ if and only if $\epsilon_k = -$.}\\
v) & \text{$\Hom_{\Bbbk Q_\epsilon}(X^\epsilon_{i,j}, X^\epsilon_{i,k}) \neq 0$ if and only if $\epsilon_k = +$.}\\
vi) & \Ext^1_{\Bbbk Q_\epsilon}(X^\epsilon_{i,k}, X^\epsilon_{i,j}) = 0 \text{ and } \Ext^1_{\Bbbk Q_\epsilon}(X^\epsilon_{i,j}, X^\epsilon_{i,k}) = 0.\\
vii) & \text{$\Hom_{\Bbbk Q_\epsilon}(X^\epsilon_{k,j}, X^\epsilon_{i,j}) \neq 0$ if and only if $\epsilon_k = +$.}\\
viii) & \text{$\Hom_{\Bbbk Q_\epsilon}(X^\epsilon_{i,j}, X^\epsilon_{k,j}) \neq 0$ if and only if $\epsilon_k = -$.}\\
ix) & \Ext^1_{\Bbbk Q_\epsilon}(X^\epsilon_{k,j}, X^\epsilon_{i,j}) = 0 \text{ and } \Ext^1_{\Bbbk Q_\epsilon}(X^\epsilon_{i,j},X^\epsilon_{k,j}) = 0.
\end{array}$
\end{lemma}

\begin{lemma}\label{separate}
Let $X^\epsilon_{k,\ell}, X^\epsilon_{i,j} \in \text{ind}(\text{rep}_\Bbbk(Q_\epsilon))$. Assume $0 \le i < j < k < \ell \le n$. Then 

$\begin{array}{rl}
i) & \text{$\Hom_{\Bbbk Q_\epsilon}(X^\epsilon_{i,j},X^\epsilon_{k,\ell}) = 0,$ $\Hom_{\Bbbk Q_\epsilon}(X^\epsilon_{k,\ell},X^\epsilon_{i,j}) = 0,$}\\
ii) & \text{$\Ext^1_{\Bbbk Q_\epsilon}(X^\epsilon_{i,j},X^\epsilon_{k,\ell}) = 0,$ $\Ext^1_{\Bbbk Q_\epsilon}(X^\epsilon_{k,\ell},X^\epsilon_{i,j}) = 0$.}
\end{array}$
\end{lemma}

Next, we present three geometric facts about pairs of distinct strands. These geometric facts will be crucial in our proof of Lemma~\ref{maintechlemma}.

\begin{lemma}\label{uniqcross}
If two distinct strands $c(i_1,j_1)$ and $c(i_2,j_2)$ on $\mathcal{S}_{n,\epsilon}$ intersect nontrivially, then $c(i_1,j_1)$ and $c(i_2,j_2)$ can be represented by a pair of {monotone} curves that have a unique transversal crossing.
\end{lemma}

\begin{proof}
{Suppose $c(i_1,j_1)$ and $c(i_2,j_2)$ intersect nontrivially. Without loss of generality, we assume $i_1 \le i_2$. Let $\gamma_k \in c(i_k,j_k)$ with $k \in [2]$ be monotone curves. There are two cases: }

$\begin{array}{rcl}
a) & i_1 \le i_2 < j_1 \le j_2\\
b) & i_1 \le i_2 < j_2 \le j_1.
\end{array}$

Suppose that case $a)$ holds. Let $(x',y') \in \{(x,y) \in \mathbb{R}^2: x_{i_2} \le x \le x_{j_1}\}$ denote a point where $\gamma_1$ crosses $\gamma_2$ transversally. If $\epsilon_{i_2} = -$ (resp. $\epsilon_{i_2} = +$), isotope $\gamma_1$ relative to $\epsilon_{i_1}$ and $(x', y')$ in such a way that the monotonocity of $\gamma_1$ is preserved and so that $\gamma_1$ lies strictly above (resp. strictly below) $\gamma_2$ on $\{(x,y) \in \mathbb{R}^2: x_{i_2} \le x < x'\}.$ 

Next, if $\epsilon_{j_1} = -$ (resp. $\epsilon_{j_1} = +$), isotope $\gamma_2$ relative to $(x',y')$ and $\epsilon_{j_2}$ in such a way that the monotonicity of $\gamma_2$ is preserved and so that $\gamma_2$ lies strictly above (resp. strictly below) $\gamma_1$ on $\{(x,y) \in \mathbb{R}^2: x' < x \le x_{j_1}\}.$ This process produces two monotone curves $\gamma_1 \in c(i_1,j_1)$ and $\gamma_2 \in c(i_2,j_2)$ that have a unique transversal crossing. The proof in case $b)$ is very similar.\end{proof}

\begin{lemma}\label{nocommonendpt}
Let $c(i_1,j_1)$ and $c(i_2,j_2)$ be distinct strands on $\mathcal{S}_{n,\epsilon}$ that intersect nontrivially. Then $c(i_1,j_1)$ and $c(i_2,j_2)$ do not share an endpoint.
\end{lemma}
\begin{proof}
Suppose $c(i_1,j_1)$ and $c(i_2,j_2)$ share an endpoint.  Thus, there exist curves $\gamma_k \in c(i_k,j_k)$ with $k \in \{1,2\}$ such that $\gamma_1$ and $\gamma_2$ are isotopic relative to their endpoints to curves with no transversal crossing.
\end{proof}

\begin{remark}\label{mononoint}
If $c(i_1, j_1)$ and $c(i_2, j_2)$ are two distinct strands on $\mathcal{S}_{n,\epsilon}$ that do not intersect nontrivially, then $c(i_1,j_1)$ and $c(i_2,j_2)$ can be represented by a pair of monotone curves $\gamma_\ell \in c(i_\ell, j_\ell)$ where $\ell \in [2]$ that are nonintersecting, except possibly at their endpoints.
\end{remark}

We now arrive at the proof of Lemma~\ref{maintechlemma}. The proof is a case by case analysis where the cases are given in terms of the entries of $\epsilon$ and the inequalities satisfied by the indices that describe a pair of indecomposable representations of $Q_\epsilon$.

\begin{proof}[Proof of Lemma~\ref{maintechlemma} a)]
Let $X^\epsilon_{i,j}:= U$ and $X^\epsilon_{k,\ell}:= V$. Assume that the strands $\Phi_{\epsilon}(X^\epsilon_{i,j})$ and $\Phi_{\epsilon}(X^\epsilon_{k,\ell})$ intersect nontrivially. By Lemma~\ref{nocommonendpt}, we can assume without loss of generality that either $0\le i < k < j < \ell \le n$ or $0 \le i < k < \ell < j \le n$. By Lemma~\ref{uniqcross}, we can represent $\Phi_{\epsilon}(X^\epsilon_{i,j})$ and $\Phi_{\epsilon}(X^\epsilon_{k,\ell})$ by monotone curves $\gamma_{i,j}$ and $\gamma_{k,\ell}$ that have a unique transversal crossing. Furthermore, we can assume that this unique crossing occurs between $\epsilon_k$ and $\epsilon_{k+1}$. There are four possible cases: 

$\begin{array}{rll}
i) & \epsilon_k = \epsilon_{k+1} = -,\\
ii) & \epsilon_k = - \text{ and } \epsilon_{k+1} = +, \\
iii) & \epsilon_k = \epsilon_{k+1} = +,\\
iv) & \epsilon_k = + \text{ and } \epsilon_{k+1} = -.
\end{array}$

\noindent We illustrate these cases up to isotopy in Figure~\ref{crossings}. We see that in cases $i)$ and $ii)$ (resp. $iii)$ and $iv)$) $\gamma_{k,\ell}$ lies 

\begin{figure}[h]
\begin{center}
\includegraphics[scale = 1.5]{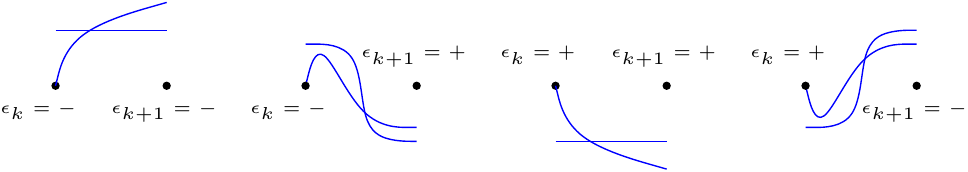}
\end{center}
\caption{The four types of crossings}
\label{crossings}
\end{figure}

\noindent above (resp. below) $\gamma_{i,j}$ inside of $\{(x,y) \in \mathbb{R}^2: x_{k+1} \le x \le x_{\text{min}\{\ell, j\}}\}$. 

Suppose $\gamma_{k,\ell}$ lies above $\gamma_{i,j}$ inside $\{(x,y) \in \mathbb{R}^2: x_{k+1} \le x \le x_{\text{min}\{\ell, j\}}\}$. Then 

$$\begin{array}{rcl}
\epsilon_{\text{min}\{\ell,j\}} & = & \left\{\begin{array}{rll} + & : & \text{min}\{\ell, j\} = \ell\\
- & : & \text{min}\{\ell,j\} = j\end{array}\right.
\end{array}$$
otherwise $\gamma_{k,\ell}$ and $\gamma_{i,j}$ would have a nonunique transversal crossing. If $\text{min}\{\ell,j\} = \ell,$ we have $0\le i < k < \ell < j \le n$, $\epsilon_k = -$, and $\epsilon_{\ell} = +.$ Now by Lemma~\ref{nested}, we have that $\Hom_{\Bbbk Q_\epsilon}(X^\epsilon_{i,j},X^\epsilon_{k,\ell}) \neq 0$ and $\Ext^1_{\Bbbk Q_\epsilon}(X^\epsilon_{k,\ell},X^\epsilon_{i,j}) \neq 0.$ If $\min\{\ell, j\} = j$, then $0 \le i < k < j < \ell \le n$, $\epsilon_k = -,$ and $\epsilon_j = -.$ Thus, by Lemma~\ref{interlaced}, we have that $\Hom_{\Bbbk Q_\epsilon}(X^\epsilon_{i,j},X^\epsilon_{k,\ell}) \neq 0$ and $\Ext^1_{\Bbbk Q_\epsilon}(X^\epsilon_{k,\ell},X^\epsilon_{i,j}) \neq 0.$

Similarly, if $\gamma_{i,j}$ lies above $\gamma_{k,\ell}$ inside $\{(x,y) \in \mathbb{R}^2: x_{k+1} \le x \le x_{\text{min}\{\ell, j\}}\}$, it follows that 

$$\begin{array}{rcl}
\epsilon_{\text{min}\{\ell,j\}} & = & \left\{\begin{array}{rll} - & : & \text{min}\{\ell, j\} = \ell\\
+ & : & \text{min}\{\ell,j\} = j.\end{array}\right.
\end{array}$$

\noindent If $\min\{\ell,j\} = \ell,$ then Lemma~\ref{nested} implies that $\Hom_{\Bbbk Q_\epsilon}(X^\epsilon_{k,\ell},X^\epsilon_{i,j}) \neq 0$ and $\Ext^1_{\Bbbk Q_\epsilon}(X^\epsilon_{i,j},X^\epsilon_{k,\ell}) \neq 0.$ If $\min\{\ell,j\} = j,$ then Lemma~\ref{interlaced} implies that $\Hom_{\Bbbk Q_\epsilon}(X^\epsilon_{k,\ell},X^\epsilon_{i,j}) \neq 0$ and $\Ext^1_{\Bbbk Q_\epsilon}(X^\epsilon_{i,j},X^\epsilon_{k,\ell}) \neq 0.$ Thus we conclude that neither $(X^\epsilon_{i,j},X^\epsilon_{k,\ell})$ nor $(X^\epsilon_{k,\ell},X^\epsilon_{i,j})$ are exceptional pairs.

Conversely, assume that neither $(U,V)$ nor $(V,U)$ are exceptional pairs where $X^\epsilon_{i,j} := U$ and $X^\epsilon_{k,\ell} := V$. Then at least one of the following is true:

$\begin{array}{rll}
a) & \text{Hom}_{\Bbbk Q_\epsilon}(X^\epsilon_{i,j}, X^\epsilon_{k,\ell}) \neq 0 \text{ and } \text{Hom}_{\Bbbk Q_\epsilon}(X^\epsilon_{k,\ell}, X^\epsilon_{i,j}) \neq 0,\\
b) & \text{Hom}_{\Bbbk Q_\epsilon}(X^\epsilon_{i,j}, X^\epsilon_{k,\ell}) \neq 0 \text{ and } \text{Ext}^1_{\Bbbk Q_\epsilon}(X^\epsilon_{k,\ell},X^\epsilon_{i,j})\neq 0, \\
c) & \text{Ext}^1_{\Bbbk Q_\epsilon}(X^\epsilon_{i,j}, X^\epsilon_{k,\ell}) \neq 0 \text{ and } \text{Hom}_{\Bbbk Q_\epsilon}(X^\epsilon_{k,\ell}, X^\epsilon_{i,j}) \neq 0,\\
d) & \text{Ext}^1_{\Bbbk Q_\epsilon}(X^\epsilon_{i,j}, X^\epsilon_{k,\ell}) \neq 0 \text{ and } \text{Ext}^1_{\Bbbk Q_\epsilon}(X^\epsilon_{k,\ell}, X^\epsilon_{i,j})\neq 0. 
\end{array}$

As $X^\epsilon_{i,j}$ and $X^\epsilon_{k,\ell}$ are indecomposable and distinct, we have that $\text{Hom}_{\Bbbk Q_\epsilon}(X^\epsilon_{i,j}, X^\epsilon_{k,\ell}) = 0$ or $\text{Hom}_{\Bbbk Q_\epsilon}(X^\epsilon_{k,\ell},X^\epsilon_{i,j}) = 0$ by Remark \ref{rem:hom0}.  Without loss of generality, assume that $\text{Hom}_{\Bbbk Q_\epsilon}(X^\epsilon_{k,\ell},X^\epsilon_{i,j}) = 0.$ Thus $b)$ or $d)$ hold so $\text{Ext}^1_{\Bbbk Q_\epsilon}(X^\epsilon_{k,\ell}, X^\epsilon_{i,j})\neq 0$. Then Lemma~\ref{shareendpt} and Lemma~\ref{separate} imply that $0 \le i < k < j < \ell \le n$ or $0 \le i < k < \ell < j \le n.$

If $0 \le i < k < j < \ell < n$, we have $\epsilon_k = \epsilon_j = -$ by Lemma~\ref{interlaced} as $\Hom_{\Bbbk Q_\epsilon}(X^\epsilon_{i,j},X_{k,\ell}^\epsilon) \neq 0$ and $\Ext^1_{\Bbbk Q_\epsilon}(X^\epsilon_{k,\ell},X^\epsilon_{i,j}) \neq 0.$  Let $\gamma_{i,j} \in \Phi_{\epsilon}(X^\epsilon_{i,j})$ and $\gamma_{k,\ell} \in \Phi_{\epsilon}(X^\epsilon_{k,\ell})$. We can assume that there exists $\delta(k) > 0$ such that $\gamma_{i,j}$ and $\gamma_{k,\ell}$ have no transversal crossing inside $\{(x,y) \in \mathbb{R}^2: x_k \le x \le x_k + \delta(k)\}.$ This implies that $\gamma_{i,j}$ lies above $\gamma_{k,\ell}$ inside $\{(x,y) \in \mathbb{R}^2: x_k\le x\le x_k + \delta(k)\}$. Similarly, we can assume there exists $\delta(j) > 0 $ such that $\gamma_{i,j}$ and $\gamma_{k,\ell}$ have no transversal crossing inside $\{(x,y) \in \mathbb{R}^2: x_j - \delta(j) \le x \le x_j\}$. Thus $\gamma_{i,j}$ lies below $\gamma_{k,\ell}$ inside $\{(x,y) \in \mathbb{R}^2: x_j - \delta(j) \le x \le x_j\}.$ This means $\gamma_{i,j}$ and $\gamma_{k,\ell}$ must have at least one transversal crossing. Thus $\Phi_{\epsilon}(X^\epsilon_{i,j})$ and $\Phi_{\epsilon}(X^\epsilon_{k,\ell})$ intersect nontrivially. An analogous argument shows that if $0 \le i < k < \ell < j \le n$, then $\Phi_{\epsilon}(X^\epsilon_{i,j})$ and $\Phi_{\epsilon}(X^\epsilon_{k,\ell})$ intersect nontrivially.
\end{proof}

\begin{proof}[Proof of Lemma~\ref{maintechlemma} b)]
Assume that $\Phi_{\epsilon}(U)$ is clockwise from $\Phi_{\epsilon}(V)$. Then we have that one of the following holds:

$\begin{array}{rll}
a) & X^\epsilon_{k,j}= U \text{ and } X^\epsilon_{i,k}= V \text{ for some } 0 \le i < k < j \le n,\\
b) & X^\epsilon_{i,k}= U \text{ and } X^\epsilon_{k,j}= V \text{ for some } 0 \le i < k < j \le n,\\
c) & X^\epsilon_{i,j}= U \text{ and } X^\epsilon_{i,k} = V \text{ for some } 0 \le i < j \le n \text{ and } 0 \le i < k \le n,\\
d) & X^\epsilon_{i,j}= U \text{ and } X^\epsilon_{k,j}= V \text{ for some } 0 \le i < j \le n \text{ and } 0 \le k < j \le n.
\end{array}$

In Case $a)$, we have that $\epsilon_k = -$ since $\Phi_{\epsilon}(X^\epsilon_{k,j})$ is clockwise from $\Phi_{\epsilon}(X^\epsilon_{i,k})$. By Lemma~\ref{shareendpt} $i)$ and $ii)$, we have that $\Hom_{\Bbbk Q_\epsilon}(X^\epsilon_{i,k},X^\epsilon_{k,j}) = 0$ and $\Ext^1_{\Bbbk Q_\epsilon}(X^\epsilon_{i,k},X^\epsilon_{k,j}) = 0.$ Thus $(X^\epsilon_{k,j}, X^\epsilon_{i,k})$ is an exceptional pair. By Lemma~\ref{shareendpt} $iii)$, we have that $\Ext^1_{\Bbbk Q_\epsilon}(X^\epsilon_{k,j}, X^\epsilon_{i,k}) \neq 0$. Thus $(X^\epsilon_{i,k},X^\epsilon_{k,j})$ is not an exceptional pair.

In Case $b)$, we have that $\epsilon_k = +$ since $\Phi_{\epsilon}(X^\epsilon_{i,k})$ is clockwise from $\Phi_{\epsilon}(X^\epsilon_{k,j})$. By Lemma~\ref{shareendpt} $i)$ and $iii)$, we have that $\Hom_{\Bbbk Q_\epsilon}(X^\epsilon_{k,j},X^\epsilon_{i,k}) = 0$ and $\Ext^1_{\Bbbk Q_\epsilon}(X^\epsilon_{k,j},X^\epsilon_{i,k}) = 0.$ Thus $(X^\epsilon_{i,k}, X^\epsilon_{k,j})$ is an exceptional pair. By Lemma~\ref{shareendpt} $ii)$, we have that $\Ext^1_{\Bbbk Q_\epsilon}(X^\epsilon_{i,k}, X^\epsilon_{k,j}) \neq 0$. Thus $(X^\epsilon_{k,j},X^\epsilon_{i,k})$ is not an exceptional pair.

In Case $c)$, if $j < k$, it follows that $\epsilon_j = -$. Indeed, since $\Phi_{\epsilon}(X^\epsilon_{i,j})$ and $\Phi_{\epsilon}(X^\epsilon_{i,k})$ share an endpoint, the two do not intersect nontrivially by Lemma~\ref{nocommonendpt}.  As $\Phi_{\epsilon}(X^\epsilon_{i,j})$ is clockwise from $\Phi_{\epsilon}(X^\epsilon_{i,k})$, Remark~\ref{mononoint} asserts that we can choose monotone curves $\gamma_{i,k} \in \Phi_{\epsilon}(X^\epsilon_{i,k})$ and $\gamma_{i,j} \in \Phi_{\epsilon}(X^\epsilon_{i,j})$ such that $\gamma_{i,k}$ lies strictly above $\gamma_{i,j}$ on $\{(x,y) \in \mathbb{R}^2: x_i < x \le x_j\}$. Thus $\epsilon_j = -.$  By Lemma~\ref{shareendpt} $v)$ and $vi)$, we have that $\Hom_{\Bbbk Q_\epsilon}(X^\epsilon_{i,k}, X^\epsilon_{i,j}) = 0$ and $\Ext^1_{\Bbbk Q_\epsilon}(X^\epsilon_{i,k}, X^\epsilon_{i,j}) = 0$ so that $(X^\epsilon_{i,j},X^\epsilon_{i,k})$ is an exceptional pair. By Lemma~\ref{shareendpt} $iv)$, we have that $\Hom_{\Bbbk Q_\epsilon}(X^\epsilon_{i,j},X^\epsilon_{i,k}) \neq 0$. Thus $(X^\epsilon_{i,k}, X^\epsilon_{i,j})$ is not an exceptional pair.

Similarly, one shows that if $k < j$, then $\epsilon_k = +$. By Lemma~\ref{shareendpt} $iv)$ and $vi)$, we have that $\Hom_{\Bbbk Q_\epsilon}(X^\epsilon_{i,k}, X^\epsilon_{i,j}) = 0$ and $\Ext^1_{\Bbbk Q_\epsilon}(X^\epsilon_{i,k}, X^\epsilon_{i,j}) = 0$ so that $(X^\epsilon_{i,j},X^\epsilon_{i,k})$ is an exceptional pair. By Lemma~\ref{shareendpt} $v)$, we have that $\Hom_{\Bbbk Q_\epsilon}(X^\epsilon_{i,j},X^\epsilon_{i,k}) \neq 0$. Thus $(X^\epsilon_{i,k}, X^\epsilon_{i,j})$ is not an exceptional pair. The proof in Case $d)$ is completely analogous to the proof in Case $c)$ so we omit it.

Conversely, let $U = X^\epsilon_{i,j}$ and $V = X^\epsilon_{k,\ell}$ and assume that $(X^\epsilon_{i,j}, X^\epsilon_{k,\ell})$ is an exceptional pair and $(X^\epsilon_{k,\ell}, X^\epsilon_{i,j})$ is not an exceptional pair. This implies that at least one of the following holds:

$\begin{array}{rll}
1) & \text{Hom}_{\Bbbk Q_\epsilon}(X^\epsilon_{k,\ell}, X^\epsilon_{i,j}) = 0, \text{Ext}^1_{\Bbbk Q_\epsilon}(X^\epsilon_{k,\ell}, X^\epsilon_{i,j}) = 0, \text{ and } \Hom_{\Bbbk Q_\epsilon}(X^\epsilon_{i,j},X^\epsilon_{k,\ell}) \neq 0,\\
2) & \text{Hom}_{\Bbbk Q_\epsilon}(X^\epsilon_{k,\ell}, X^\epsilon_{i,j}) = 0, \text{Ext}^1_{\Bbbk Q_\epsilon}(X^\epsilon_{k,\ell}, X^\epsilon_{i,j}) = 0, \text{ and } \Ext^1_{\Bbbk Q_\epsilon}(X^\epsilon_{i,j},X^\epsilon_{k,\ell}) \neq 0.
\end{array}$

\noindent By Lemma~\ref{separate}, we know that $[i,j]\cap[k,\ell] \neq \emptyset.$ This implies that either

$\begin{array}{rll}
i) & \Phi_{\epsilon}(X^\epsilon_{i,j}) \text{ and } \Phi_{\epsilon}(X^\epsilon_{k,\ell}) \text{ share an endpoint},\\
ii) & 0 \le i < k < j < \ell \le n,\\
iii) & 0 \le i < k < \ell < j \le n, \\
iv) & 0 \le k < i < \ell < j \le n,\\
v) & 0 \le k < i < j < \ell \le n.
\end{array}$

\noindent We will show that $\Phi_{\epsilon}(X^\epsilon_{i,j})$ and $\Phi_{\epsilon}(X^\epsilon_{k,\ell})$ share an endpoint. 

Suppose $0 \le i < k < j < \ell \le n.$ Then since $\text{Hom}_{\Bbbk Q_\epsilon}(X^\epsilon_{k,\ell}, X^\epsilon_{i,j}) = 0, \text{Ext}^1_{\Bbbk Q_\epsilon}(X^\epsilon_{k,\ell}, X^\epsilon_{i,j}) = 0$, we have by Lemma~\ref{interlaced} $ii)$ and $iv)$ that either $\epsilon_k = -$ and $\epsilon_j = +$ or $\epsilon_{k} = +$ and $\epsilon_j = -$. However, as $\Hom_{\Bbbk Q_\epsilon}(X^\epsilon_{i,j},X^\epsilon_{k,\ell}) \neq 0$ or $\Ext^1_{\Bbbk Q_\epsilon}(X^\epsilon_{i,j},X^\epsilon_{k,\ell}) \neq 0$, Lemma~\ref{interlaced} $i)$ and $iii)$ assert that $\epsilon_k = \epsilon_j = -$ or $\epsilon_k = \epsilon_j = +.$ This is a contradiction.  Thus, $i, j, k, \ell$ do not satisfy $0 \leq i < k < j < \ell \leq n$, and by a similar argument, they also do not satisfy $0 \le k < i < \ell < j \le n$.

Suppose $0 \le i < k < \ell < j \le n.$ Then since $\text{Hom}_{\Bbbk Q_\epsilon}(X^\epsilon_{k,\ell}, X^\epsilon_{i,j}) = 0, \text{Ext}^1_{\Bbbk Q_\epsilon}(X^\epsilon_{k,\ell}, X^\epsilon_{i,j}) = 0$, we have by Lemma~\ref{nested} $ii)$ and $iv)$ that either $\epsilon_k =\epsilon_\ell = +$ or $\epsilon_k = \epsilon_\ell = -$. However, as  $\Hom_{\Bbbk Q_\epsilon}(X^\epsilon_{i,j},X^\epsilon_{k,\ell}) \neq 0$ or $\Ext^1_{\Bbbk Q_\epsilon}(X^\epsilon_{i,j},X^\epsilon_{k,\ell}) \neq 0$, Lemma~\ref{nested} $i)$ and $iii)$ we have that $\epsilon_k = -$ and $\epsilon_\ell = +$ or $\epsilon_k = +$ and $\epsilon_\ell = -.$ This is a contradiction.  Thus, $i,j,k,\ell$ do not satisfy $0 \le i < k < \ell < j \le n,$ and by an analogous argument, they also do not satisfy $0 \le k < i < j < \ell \le n.$

We conclude that $\Phi_{\epsilon}(U)$ and $\Phi_{\epsilon}(V)$ share an endpoint. Thus we have that one of the following holds where we forget the previous roles played by $i,j,k$:

$\begin{array}{rll}
a) & X^\epsilon_{k,j}= U \text{ and } X^\epsilon_{i,k}= V \text{ for some } 0 \le i < k < j \le n,\\
b) & X^\epsilon_{i,k}= U \text{ and } X^\epsilon_{k,j}= V \text{ for some } 0 \le i < k < j \le n,\\
c) & X^\epsilon_{i,j}= U \text{ and } X^\epsilon_{i,k} = V \text{ for some } 0 \le i < j \le n \text{ and } 0 \le i < k \le n,\\
d) & X^\epsilon_{i,j}= U \text{ and } X^\epsilon_{k,j}= V \text{ for some } 0 \le i < j \le n \text{ and } 0 \le k < j \le n.
\end{array}$

Suppose Case $a)$ holds. Then since $(U,V)$ is an exceptional pair, we have $\Ext^1_{\Bbbk Q_\epsilon}(X^\epsilon_{i,k}, X^\epsilon_{k,j}) = 0.$ By Lemma~\ref{shareendpt} $ii)$, we have that $\epsilon_k = -$. Thus $\Phi_{\epsilon}(U)$ is clockwise from $\Phi_{\epsilon}(V)$.

Suppose Case $b)$ holds. Then since $(U,V)$ is an exceptional pair, we have $\Ext^1_{\Bbbk Q_\epsilon}(X^\epsilon_{k,j}, X^\epsilon_{i,k}) = 0.$ By Lemma~\ref{shareendpt} $iii)$, we have that $\epsilon_k = +$. Thus $\Phi_{\epsilon}(U)$ is clockwise from $\Phi_{\epsilon}(V)$.

Suppose Case $c)$ holds. Assume $k < j$. Then Lemma~\ref{shareendpt} $iv)$ and the fact that $\Hom_{\Bbbk Q_\epsilon}(X^\epsilon_{i,k},X^\epsilon_{i,j}) = 0$ imply that $\epsilon_k = +.$ Thus we have that $\Phi_{\epsilon}(U) = \Phi_{\epsilon}(X^\epsilon_{i,j})$ is clockwise from $\Phi_{\epsilon}(V) = \Phi_{\epsilon}(X^\epsilon_{i,k})$. Now suppose $j < k$. Then Lemma~\ref{shareendpt} $v)$ and $\Hom_{\Bbbk Q_\epsilon}(X^\epsilon_{i,k}, X^\epsilon_{i,j}) = 0$ imply that $\epsilon_j = -.$ Thus we have that $\Phi_{\epsilon}(U) = \Phi_{\epsilon}(X^\epsilon_{i,j})$ is clockwise from $\Phi_{\epsilon}(V) = \Phi_{\epsilon}(X^\epsilon_{i,k})$. The proof in Case $d)$ is very similar so we omit it.
\end{proof}

\begin{proof}[Proof of Lemma~\ref{maintechlemma} $c)$]
Observe that two strands $c(i_1,j_1)$ and $c(i_2, j_2)$ share an endpoint if and only if one of the two strands is clockwise from the other. Thus Lemma~\ref{maintechlemma} $a)$ and $b)$ implies that $\Phi_{\epsilon}(U)$ and $\Phi_{\epsilon}(V)$ do not intersect at any of their endpoints and they do not intersect nontrivially if and only if both $(U,V)$ and $(V,U)$ are exceptional pairs. 
\end{proof}

\section{Mixed cobinary trees}\label{sec:mct}

We recall the definition of an $\epsilon$-mixed cobinary tree and construct a bijection between the set of (isomorphism classes of) such trees and the set of maximal oriented strand diagrams on $\mathcal{S}_{n,\epsilon}.$

\begin{definition}[\cite{io13}]\label{def: MCT}
Given a sign function $\epsilon:[0,n]\to\{+,-\}$, an \textbf{$\epsilon$-mixed cobinary tree} (MCT) is a tree $T$ embedded in $\mathbb R^2$ with vertex set $\{(i,y_i): i\in[0,n]\}$ and edges straight line segments and satisfying the following conditions.

$\begin{array}{rl}
a) & \text{None of the edges is horizontal.}\\
b) & \text{If $\epsilon_i=+$ then $y_i\ge z$ for any $(i,z)\in T$. So, the tree goes under $(i,y_i)$.}\\
c) & \text{If $\epsilon_i=-$ then $y_i\le z$ for any $(i,z)\in T$. So, the tree goes over $(i,y_i)$.}\\
d) & \text{If $\epsilon_i=+$ then there is at most one edge descending from $(i,y_i)$ and}\\
&\text{at most two edges ascending from $(i,y_i)$ and not on the same side.}\\
e) & \text{If $\epsilon_i=-$ then there is at most one edge ascending from $(i,y_i)$ and}\\
&\text{at most two edges descending from $(i,y_i)$ and not on the same side.}\\
\end{array}$

Two MCTs $T,T'$ are \textbf{isomorphic} as MCTs if there is a graph isomorphism $T\cong T'$ which sends $(i,y_i)$ to $ (i,y_i')$ and so that corresponding edges have the same sign of their slopes.
\end{definition}

Given a MCT $T$, there is a partial ordering on $[0,n]$ given by $i<_Tj$ if the unique path from $(i,y_i)$ to $(j,y_j)$ in $T$ is monotonically increasing. Isomorphic MCTs give the same partial ordering by definition. Conversely, the partial ordering $<_T$ determines $T$ uniquely up to isomorphism since $T$ is the Hasse diagram of the partial ordering $<_T$. We sometimes decorate MCTs with \textbf{leaves} at vertices so that the result is \textbf{trivalent}, i.e., with three edges incident to each vertex. See, e.g., Figure \ref{figMCTa}. The ends of these leaves are not considered to be vertices. In that case, each vertex with $\epsilon=+$ forms a ``Y'' and this pattern is vertically inverted for $\epsilon=-$. The position of the leaves is uniquely determined.

\begin{figure}[ht]\label{fig:MCT example}
\begin{minipage}[b]{0.45\textwidth}
\begin{tikzpicture}
\begin{scope}
\draw (-1.8,1) node{ };
	\draw[ thick,color=blue] (0,0)--(1,1)--(3,0)--(4,1)(2,-1)--(3,0);
	\draw[fill] (4,1) circle[radius=2pt] node[right]{$(4,1)$} 
	(1,1)  circle[radius=2pt]node[above]{$(1,1)$}
	(0,0)  circle[radius=2pt]node[left]{$(0,0)$} 
	 (3,0)circle[radius=2pt](3,-.2)node[right]{$(3,0)$} 
	(2,-1) circle[radius=2pt]node[left]{$(2,-1)$};
\end{scope}
\end{tikzpicture}
	\caption{A MCT with $\epsilon_1=\epsilon_2=-$, $\epsilon_3=+$ and any value for $\epsilon_0$, $\epsilon_4$.}
	\label{figMCTb}
	\end{minipage}
	\begin{minipage}[b]{0.45\textwidth}
\begin{tikzpicture}[yscale=.75]
\draw (-2.3,1) node{ };
	\draw[ thick,color=blue] (0,0)--(1,1)--(2,2)--(3,3);
	\foreach\y in {0,2,3}\draw[thick,color=green] (\y,\y)--+(.4,-.4);
	\draw[thick,color=green] (3,3)--+(.4,.4) (0,0)--(-.4,-.4) (1,1)--+(-.4,.4);
	\foreach\x in {0,...,3}\draw[fill] (\x,\x) ellipse[x radius=2pt,y radius=3pt];
\end{tikzpicture}
	\caption{This MCT (in blue) has added green leaves showing that $\epsilon=(-,+,-,-)$.}
	\label{figMCTa}
	\end{minipage}
\end{figure}

In Figure \ref{figMCTa}, the four vertices have coordinates $(0,y_0),(1,y_1),(2,y_2),(3,y_3)$ where the $y_i$ can be any real numbers such that $y_0<y_1<y_2<y_3$. This inequality defines an open subset of $\mathbb R^4$ which is called the {region} of this tree $T$. More generally, for any MCT $T$, the \textbf{region} of $T$, denoted $\mathcal R_\epsilon(T)$, is the set of all points $y\in\mathbb R^{n+1}$ with the property that there exists a mixed cobinary tree $T'$ which is isomorphic to $T$ so that the vertex set of $T'$ is $\{(i,y_i): i\in[0,n]\}$.

\begin{theorem}[\cite{io13}]
Let $n \in \ZZ_{\geq 0}$ and $\epsilon:[0,n]\to \{+,-\}$ be fixed. Then, for every MCT $T$, the region $\mathcal R_\epsilon(T)$ is convex and nonempty. Furthermore, every point $y=(y_0,\ldots,y_n)$ in $\mathbb R^{n+1}$ with distinct coordinates lies in $\mathcal R_\epsilon(T)$ for a unique $T$ (up to isomorphism). In particular these regions are disjoint and their union is dense in $\mathbb R^{n+1}$.
\end{theorem}

For a fixed $n$ and $\epsilon:[0,n]\to \{+,-\}$ we will construct a bijection between the set $\mathcal T_\epsilon$ of isomorphism classes of mixed cobinary trees with sign function $\epsilon$ and the set $\overrightarrow{\mathcal{D}}_{n,\epsilon}$ defined in Definition~\ref{def:cmatdiag}.

\begin{lemma}\label{lem: orientation of paths in OSDs}
Let $\overrightarrow{d} = \{\overrightarrow{c}(i_\ell, j_\ell)\}_{\ell \in [n]} \in \overrightarrow{\mathcal{D}}_{n,\epsilon}$ and let $\overrightarrow{\gamma}_\ell \in \overrightarrow{c}(i_\ell, j_\ell)$ with $\ell \in [n]$ be monotone paths from $i_\ell$ to $j_\ell$ that are pairwise non-intersecting except possibly at their endpoints. If $p, q$ are two points on the union of these paths where $q$ lies directly above $p$, then the paths $\overrightarrow{\gamma}_{i_1}, \ldots, \overrightarrow{\gamma}_{i_k}$ appearing in the unique curve $\overrightarrow{\gamma}$ connecting $p$ to $q$ are oriented in the same direction as $\overrightarrow{\gamma}$.
\end{lemma}

\begin{proof}
The proof will be by induction on the number $m$ of internal vertices of the path $\gamma$. If $m=1$ with internal vertex $\epsilon_i$ then the path $\gamma$ has only two edges of $\overrightarrow d$: one going from $p$ to $\epsilon_i$, say to the left, and the other going from $\epsilon_i$ back to $q$. Since $\overrightarrow{d} \in \overrightarrow{\mathcal{D}}_{n,\epsilon}$, the edge coming into $\epsilon_i$ from its right is below the edge going out from $\epsilon_i$ to $q$. Therefore the orientation of these two edges in $\overrightarrow{d}$ matches that of $\gamma$.

Now suppose that $m\ge2$ and the lemma holds for smaller $m$. There are two cases. Case 1: The path $\gamma$ lies entirely on one side of $p$ and $q$ (as in the case $m=1$). Case 2: $\gamma$ has internal vertices on both sides of $p,q$.

\underline{Case 1}: Suppose by symmetry that $\gamma$ lies entirely on the left side of $p$ and $q$. Let $j$ be maximal so that $\epsilon_j$ is an internal vertex of $\gamma$. We claim that $\epsilon_j$ cannot be a local maximum of the ($x$-coordinate on the) curve $\gamma$ since, if it were, it would need to be either both below or both above the paths from $\epsilon_j$ to $p$ and to $q$. But, in that case, we can apply induction on $m$ to conclude that the edges in $\gamma$ adjacent to $\epsilon_j$ are oriented in the same direction contradicting the definition of $\overrightarrow{\mathcal{D}}_{n,\epsilon}$. Thus, $\gamma$ contains an edge connecting $\epsilon_j$ to either $p$ or $q$, say $p$. And the edge of $\gamma$ ending in $q$ contains a unique point $r$ which lies directly above $\epsilon_j$. This forces the sign to be $\epsilon_j=-$. By induction on $m$, we can conclude that each edge of $\overrightarrow d$ is oriented in the same direction as the path from $\epsilon_j$ to $r$.  So, it must be oriented outward from $\epsilon_j$. Since $\epsilon_j=-$, any other edge at $\epsilon_j$ is oriented inward. So, the edge between $p$ and $\epsilon_j$ is oriented from $p$ to $\epsilon_j$ as required. The edge coming into $r$ from the left is oriented to the right (by induction). So, this same edge continues to be oriented to the right as it goes from $r$ to $q$. The other subcases (when $\epsilon_j$ is connected to $q$ instead of $p$ and when $\gamma$ lies to the right of $p$ and $q$) are analogous.

\underline{Case 2}: Suppose that $\gamma$ is on both sides of $p$ and $q$. Then $\gamma$ passes through a third point, say $r$, on the same vertical line containing $p$ and $q$. Let $\gamma_0$ and $\gamma_1$ denote the parts of $\gamma$ going from $p$ to $r$ and from $r$ to $q$ respectively. Then $\gamma_0,\gamma_1$ each have at least one internal vertex. So, the lemma holds for each of them separately. There are three subcases: either (a) $r$ lies below $p$, (b) $r$ lies above $q$ or (c) $r$ lies between $p$ and $q$. In subcase (a), we have, by induction on $m$, that $\gamma_0,\gamma_1$ are both oriented away from $r$. So, $r=\epsilon_k=+$ which contradicts the assumption that $q$ lies above $r$. Similarly, subcase (b) is not possible. In subcase (c), we have by induction on $m$ that the orientations of the edges of $\overrightarrow d$ are compatible with the orientations of $\gamma_0$ and $\gamma_1$. So, the lemma holds in  subcase (c), which is the only subcase of Case 2 which is possible. Therefore, the lemma holds in all cases.  
\end{proof}

\begin{theorem}\label{thm: bijection between MCTs and OSDs}
For each $\overrightarrow{d} = \{\overrightarrow{c}(i_\ell,j_\ell)\}_{\ell \in [n]} \in \overrightarrow{\mathcal{D}}_{n,\epsilon}$, let $\mathcal R(\overrightarrow{d})$ denote the set of all $y\in \mathbb R^{n+1}$ so that $y_i<y_j$ for any $\overrightarrow c(i,j)$ in $\overrightarrow{d}$.  Then $\mathcal R(\overrightarrow{d})=\mathcal R_\epsilon(T)$ for a uniquely determined mixed cobinary tree $T\in \mathcal T_\epsilon$. Furthermore, this gives a bijection
\[
\overrightarrow{\mathcal{D}}_{n,\epsilon}\cong \mathcal T_\epsilon.
\]
\end{theorem}

\begin{proof} We first verify the existence of a mixed cobinary tree $T$ for every choice of $y\in\mathcal R(\overrightarrow d)$. Since the strand diagram is a tree, the vector $y$ is uniquely determined by $y_0\in\mathbb R$ and $y_{j_\ell}-y_{i_\ell}>0$, $\ell\in [n]$, which are arbitrary. Given such a $y$, we claim that the $n$ line segments $L_\ell$ in $\mathbb R^2$ connecting the pairs of points $(i_\ell,y_{i_\ell}), (j_\ell,y_{j_\ell})$ meet only {at} their endpoints. If not then two of these line segments, say $L_s,L_t$, meet at some point $(a,b)\in \mathbb R^2$. This leads to a contradiction of Lemma \ref{lem: orientation of paths in OSDs} as follows. Choose representatives $\overrightarrow{\gamma}_\ell \in \overrightarrow{c}(i_\ell, j_\ell)$ as in the lemma. Let $p\in \overrightarrow{\gamma}_s$ and $q\in \overrightarrow{\gamma}_t$ be the points on those curves with $x$-coordinate $a$. By symmetry assume $p$ is below $q$. Let $\overrightarrow{\gamma}_{w_1}, \ldots, \overrightarrow{\gamma}_{w_k}$ be the monotone curves in the unique path $\overrightarrow\gamma$ connecting $p$ and $q$ so that $w_1=s$ and $w_k=t$. By Lemma \ref{lem: orientation of paths in OSDs} these curves are oriented in the same direction as $\overrightarrow\gamma$. By definition of the vector $y\in \mathcal R(\overrightarrow{d})$ we have $y_{i_\ell}<y_{j_\ell}$ for each $\ell=w_1,\ldots,w_k$. Then $b<y_{j_s}<y_{i_t}<b$ is a contradiction.
So, $T$ is a linearly embedded tree. The lemma also implies that the tree $T$ lies above all negative vertices and below all positive vertices. The other parts of Definition \ref{def: MCT} follow from the definition of an oriented strand diagram. Therefore $T\in \mathcal T_\epsilon$. Since this argument works for every $y\in\mathcal R(\overrightarrow d)$, we see that $\mathcal R(\overrightarrow d)=\mathcal R_\epsilon(T)$ as claimed.

A description of the inverse mapping $\mathcal T_\epsilon\to \overrightarrow{\mathcal{D}}_{n,\epsilon}$ is given as follows. Take any $\epsilon$-MCT $T$ and deform the tree by moving all vertices vertically to the subset $[0,n]\times 0$ on the $x$-axis and deforming the edges in such a way that they are always embedded in the plane with no vertical tangents and so that their interiors do not meet. The result is an oriented strand diagram $\overrightarrow d$ with $\mathcal R(\overrightarrow d)=\mathcal R_\epsilon(T)$.

It is clear that these are inverse mappings giving the desired bijection $\overrightarrow{\mathcal{D}}_{n,\epsilon}\cong \mathcal T_\epsilon$. 
\end{proof}

\begin{example}
The MCTs in Figures \ref{figMCTb} and \ref{figMCTa} above give the oriented strand diagrams:

\begin{center}
\begin{tikzpicture}
\begin{scope}
	\draw[thick,color=blue,-stealth]
	(0,0) .. controls (0.2,1) and (.8,-1)..(.97,-.07);
	\draw[thick,color=blue,-stealth]
	(1,0) .. controls (1.2,-1) and (1.8,1)..(1.97,.07);
	\draw[thick,color=blue,-stealth]
	(2,0) .. controls (2.2,.5) and (2.8,.5)..(2.97,.07);
	\foreach\x in {0,1,2,3}\draw[fill] (\x,0) circle[radius=2pt];
\end{scope}
\begin{scope}[xshift=-7cm]
	\draw[thick,color=blue,-stealth]
	(0,0) .. controls (0.2,.5) and (.8,.5)..(.97,.07);
	\draw[thick,color=blue,-stealth]
	(2,0) .. controls (2.2,1) and (2.8,-1)..(3,-.07);
	\draw[thick,color=blue,-stealth]
	(3,0) .. controls (3.2,-.5) and (3.8,-.5)..(3.97,-.07);
	\draw[thick,color=blue,-stealth]
	(3,0) .. controls (2.5,-1) and (3,1.5)..(1.03,.07);
	\foreach\x in {0,...,4}\draw[fill] (\x,0) circle[radius=2pt];
\end{scope}
\end{tikzpicture}
\end{center}

and the oriented strand diagram in Example \ref{thirdexample} gives the MCT:

\begin{center}
\begin{tikzpicture}
\begin{scope}[xshift=-7cm]
	\draw[thick,color=blue,-stealth]
	(0,0) .. controls (1.6,-1.3) and (1.7,.7)..(1.97,.07);
	\draw[thick,color=blue,-stealth]
	(2,0) .. controls (2.2,1) and (2.8,-1)..(3,-.07);
	\draw[thick,color=blue,-stealth]
	(3,0) .. controls (3.2,-.5) and (3.8,-.5)..(3.97,-.07);
	\draw[thick,color=blue,-stealth]
	(3,0) .. controls (2.7,-.7) and (3,.5)..(2,.5)
	..controls (1,.5) and (1.3,-.7)..(1.03,-.07);
	\foreach\x in {0,...,4}\draw[fill] (\x,0) circle[radius=2pt];
\end{scope}
\begin{scope}
\draw (-1.5,0) node{$\iff$};
	\draw[ thick,color=blue,-stealth] 
	(0,-.5)--(1.95,-.03);
	\draw[ thick,color=blue,-stealth] 
	(2,0)--(2.95,.48);
	\draw[ thick,color=blue,-stealth] 
	(3,.5)--(3.95,.98);
	\draw[ thick,color=blue,-stealth] 
	(3,.5)--(1.05,.98);
	\draw[fill] (4,1) circle[radius=2pt] 
	(1,1)  circle[radius=2pt]
	(0,-.5)  circle[radius=2pt] 
	 (3,.5)circle[radius=2pt](3,-.2) 
	(2,0) circle[radius=2pt];
\end{scope}
\end{tikzpicture}
\end{center}
\end{example}

We now arrive at the proof of Theorem~\ref{c-matClassif}.  This theorem follows from the fact that oriented diagrams belonging to $\overrightarrow{\mathcal{D}}_{n,\epsilon}$ can be regarded as mixed cobinary trees by Theorem~\ref{thm: bijection between MCTs and OSDs}.

\begin{proof}[Proof of Theorem~\ref{c-matClassif}]
Let $f$ be the map $\textbf{c}\text{-mat}(Q_\epsilon)\to \overrightarrow{\mathcal{D}}_{n,\epsilon}$ induced by the map defined in Lemma \ref{cmatdiag}, and let $g$ be the bijective map $\mathcal T_\epsilon\to \overrightarrow{\mathcal{D}}_{n,\epsilon}$ defined in Theorem~\ref{thm: bijection between MCTs and OSDs}.  We will assert the existence of a map $h:\textbf{c}\text{-mat}(Q_\epsilon) \to \mathcal T_\epsilon$ which fits into the diagram

\begin{center}
\begin{tikzcd}
\textbf{c}\text{-mat}(Q_\epsilon) \ar{dr}[swap]{f} \ar{rr}{h} & & \mathcal T_\epsilon \ar{dl}{g}[swap]{\sim} \\
& \overrightarrow{\mathcal{D}}_{n,\epsilon} &
\end{tikzcd}
\end{center}
The theorem will follow after verifying that $h$ is a bijection and that $f = g \circ h$.

We will define two new notions of $\textbf{c}\text{-matrix}$, one for MCTs and one for oriented strand diagrams.  Let $T \in \mathcal T_\epsilon$ with internal edges $\ell_i$ having endpoints $(i_1,y_{i_1})$ and $(i_2,y_{i_2})$.  For each $\ell_i$, define the `$\textbf{c}\text{-vector}$' of $\ell_i$ to be $c_i(T) := \sum_{i_1< j \le i_2} \text{sgn}(\ell_i)e_j$, where $\text{sgn}(\ell_i)$ is the sign of the slope of $\ell_i$.  Define $c(T)$ to be the `$\textbf{c}\text{-matrix}$' of $T$ whose rows are the $\textbf{c}\text{-vectors } c_i(T)$.  Now, let $\overrightarrow{d} = \{\overrightarrow{c}(i_\ell,j_\ell)\}_{\ell \in [n]} \in \overrightarrow{\mathcal{D}}_{n,\epsilon}$.  For each oriented strand $\overrightarrow{c}(i_\ell,j_\ell)$, define the `$\textbf{c}\text{-vector}$' of $\overrightarrow{c}(i_\ell,j_\ell)$ to be 
$$\begin{array}{rcl}
c_\ell(\overrightarrow{d}) & := & \left\{\begin{array}{rll} \sum_{i_\ell< k \le j_\ell} \text{sgn}(\overrightarrow{c}(i_\ell,j_\ell))e_k & : & i_\ell < j_\ell\\
\sum_{j_\ell< k \le i_\ell} \text{sgn}(\overrightarrow{c}(i_\ell,j_\ell))e_k & : & i_\ell > j_\ell\end{array}\right.
\end{array}$$
where $\text{sgn}(\overrightarrow{c}(i_\ell,j_\ell))$ is positive if $i_\ell < j_\ell$ and negative if $i_\ell > j_\ell$.  Define $c(\overrightarrow{d})$ to be the `$\textbf{c}\text{-matrix}$' of $\overrightarrow{d}$ whose rows are the $\textbf{c}\text{-vectors } c_\ell(\overrightarrow{d})$.

It is known that the notion of $\textbf{c}\text{-matrix}$ for MCTs coincides with the original notion of $\textbf{c}\text{-matrix}$ defined in Section~\ref{subsec:quivers}, and that there is a bijection between $\textbf{c}\text{-mat}(Q_\epsilon)$ and $\mathcal T_\epsilon$ which preserves $\textbf{c}\text{-matrices}$ (see \cite[Remarks 2 and 4]{io13} for details).  Thus, we have a bijective map $h: \textbf{c}\text{-mat}(Q_\epsilon)\to \mathcal T_\epsilon$.  On the other hand, the bijection $g : \mathcal T_\epsilon\to \overrightarrow{\mathcal{D}}_{n,\epsilon}$ defined in Theorem~\ref{thm: bijection between MCTs and OSDs} also preserves $\textbf{c}\text{-matrices}$.  The map $f: \textbf{c}\text{-mat}(Q_\epsilon) \to \mathcal{T}_\epsilon$ preserves $\textbf{c}\text{-matrices}$ by definition.  Hence, we have $f = g \circ h$ and $f$ is a bijection, as desired.
\end{proof}

\begin{remark}
 For linearly-ordered quivers (those with $\epsilon = (+,\ldots,+)$ or $\epsilon = (-,\ldots,-))$, this bijection was established by the first and third authors in \cite{gm1} using a different approach.  The bijection was given by hand without going through MCTs.  This was more tedious, and the authors feel that some aspects (such as mutation) are better phrased in terms of MCTs.
\end{remark}

\section{Exceptional sequences and linear extensions}\label{sec:pos}
In this section, we consider the problem of counting the number of CESs arising from a given CEC. We show that this problem can be restated as the problem of counting the number of linear extensions of certain posets. 

\begin{definition}
We define the \textbf{poset} $\mathcal{P}_d = (\{c(i_\ell,j_\ell)\}_{\ell \in [n]}, \le)$ \textbf{associated to $d$} as the partially ordered set whose elements are the strands of $d$ and where $c(k,\ell)$ covers $c(i,j)$, denoted by $c(i,j) \lessdot c(k,\ell)$, if and only if the strand $c(k,\ell)$ is clockwise from $c(i,j)$ and there does not exist another strand $c(i^\prime,j^\prime)$ distinct from $c(i,j)$ and $c(k,\ell)$ such that  $c(i^\prime,j^\prime)$ is clockwise from $c(i,j)$ and counterclockwise from $c(k,\ell)$. 
\end{definition}

The construction defines a poset because any oriented cycle in the Hasse diagram of $\mathcal{P}_d$ arises from a cycle in the graph determined by $d$. Since the  graph determined by $d$ is a tree, it has no cycles.  In Figure~\ref{fig:PosetExample}, we show a diagram $d \in \mathcal{D}_{4,\epsilon}$ where $\epsilon :=(-,+,-,+,+)$ and its poset $\mathcal{P}_d.$

\begin{figure}[ht]
\centering
\raisebox{.3in}{\begin{minipage}[b]{0.4\textwidth}
\includegraphics[width=\textwidth]{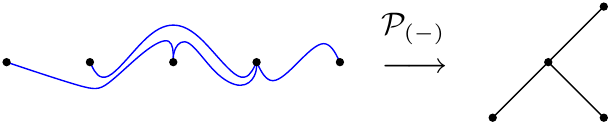}
\caption{A diagram and its poset.}
\label{fig:PosetExample}
\end{minipage}}
\ \ \ \ \ \ \ \ \ \ \
\begin{minipage}[b]{0.4\textwidth}
\includegraphics[width=\textwidth]{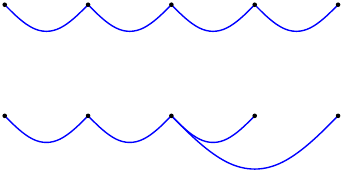}
\caption{Two diagrams with the same poset.}
\label{fig:SamePosets}
\end{minipage}
\end{figure}

In general, the map $\mathcal{D}_{n,\epsilon} \rightarrow \mathscr{P}(\mathcal{D}_{n,\epsilon}):=\{\mathcal{P}_d: d \in \mathcal{D}_{n,\epsilon}\}$ is not injective. For instance, each of the two diagrams in Figure~\ref{fig:SamePosets} have $\mathcal{P}_d = \textbf{4}$ where $\textbf{4}$ denotes the linearly-ordered poset with 4 elements. It is thus natural to ask which posets are obtained from strand diagrams. 

Our next result describes the posets arising from diagrams in $\mathcal{D}_{n,\epsilon}$ where $\epsilon = (-,\ldots,-)$ or $\epsilon = (+,\ldots,+)$. Before we state it, we remark that diagrams in $\mathcal{D}_{n,\epsilon}$ where $\epsilon = (-,\ldots,-)$ or $\epsilon = (+,\ldots,+)$ can be regarded as \textbf{chord diagrams}.\footnote{These noncrossing trees embedded in a disk with vertices lying on the boundary have been studied by Araya in \cite{a13}, Goulden and Yong in \cite{gy02}, and the first and third authors in \cite{gm1}.} Figure~\ref{ident_chord_strand} gives an example of this identification. Under this identification, the term strand is synonymous with \textbf{chord}. 

\begin{figure}[h]
$$\includegraphics[scale = .75]{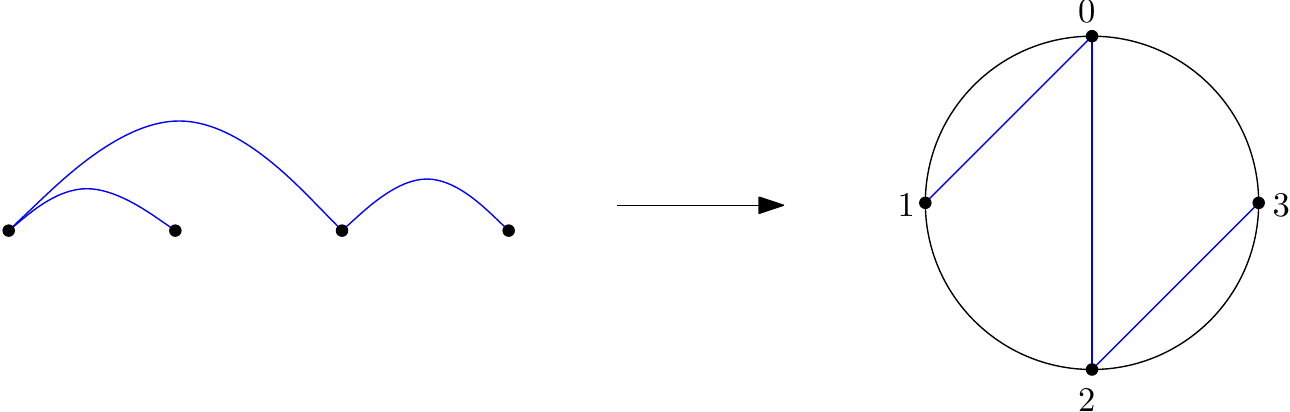}$$
\caption{The identification of between chord diagrams and strand diagrams.}
\label{ident_chord_strand}
\end{figure}

Let $d \in \mathcal{D}_{n,\epsilon}$ where $\epsilon = (-,\ldots,-)$ or $\epsilon = (+,\ldots,+)$. Let $c(i,j)$ be a strand of $d$. There is an obvious action of $\mathbb{Z}/(n+1)\mathbb{Z}$ on chord diagrams. Let $\tau \in \mathbb{Z}/(n+1)\mathbb{Z}$ denote a generator and define $\tau c(i,j) := c(i-1, j-1)$ and $\tau^{-1} c(i,j) := c(i+1, j+1)$ where we consider $i\pm 1$ and $j\pm 1$ mod $n+1$. We also define $\tau d := \{\tau c(i_\ell,j_\ell)\}_{\ell \in [n]}$ and $\tau^{-1}d := \{\tau^{-1} c(i_\ell,j_\ell)\}_{\ell \in [n]}$. The next lemma, which is easily verified, shows that the order-theoretic properties of CECs are invariant under the action of $\tau^{\pm 1}$.

\begin{lemma}\label{rotationinv}
Let $d \in \mathcal{D}_{n,\epsilon}$ where $\epsilon = (-,\ldots,-)$ or $\epsilon = (+,\ldots,+)$. Then we have the following isomorphisms of posets $\mathcal{P}_d \cong \mathcal{P}_{\tau d}$ and $\mathcal{P}_d \cong \mathcal{P}_{\tau^{-1}d}.$
\end{lemma}

\begin{theorem}\label{posetclassif}
Let $\epsilon = (-,\ldots,-)$ or let $\epsilon = (+,\ldots, +).$ Then a poset $\mathcal{P} \in \mathscr{P}(\mathcal{D}_{n,\epsilon})$ if and only if 

$\begin{array}{rllc}
i) & \text{each $x \in \mathcal{P}$ has at most two covers and covers at most two elements,}\\
ii) & \text{the underlying graph of the Hasse diagram of $\mathcal{P}$ has no cycles,}\\
iii) & \text{the Hasse diagram of $\mathcal{P}$ is connected}.\\
\end{array}$
\end{theorem}
\begin{proof}
Let $\mathcal{P}_d \in \mathscr{P}(\mathcal{D}_{n,\epsilon})$. By definition, $\mathcal{P}_d$ satisfies $i)$. It is also clear that the Hasse diagram of $\mathcal{P}_d$ is connected since the graph determined by $d$ is connected. To see that $\mathcal{P}_d$ satisfies $ii)$, {suppose that $C$ is a full subposet of $\mathcal{P}_d$ whose Hasse diagram is a \textbf{minimal cycle} (i.e. the underlying graph of $C$ is a cycle, but does not contain a proper subgraph that is a cycle).} Thus there exists $c_0 \in C$ that is covered by two distinct elements $c_1, c_\ell \in C$ in $\mathcal{P}_d$ where $\ell \le n$. Observe that $C$ can be regarded as a sequence of distinct chords $\{c_i\}_{i = 0}^\ell$ of $d$ where for all $i \in [0,\ell],$ $c_i$ and $c_{i+1}$ (we consider the indices modulo $\ell + 1$) share an endpoint $j$, and no chord adjacent to $j$ appears between $c_i$ and $c_{i+1}$. Observe that the graph determined by the diagram $d \backslash c_0$ has two connected components. We conclude that such a sequence $\{c_i\}_{i = 0}^\ell$ cannot exist in $d$. Thus the Hasse diagram of $\mathcal{P}_d$ has no cycles.


To prove the converse, we proceed by induction on the number of elements of $\mathcal{P}$ where $\mathcal{P}$ is a poset satisfying conditions $i), ii), iii)$. If $\#\mathcal{P} = 1$, then $\mathcal{P}$ is the unique poset with one element and $\mathcal{P} = \mathcal{P}_d$ where $d$ is the unique chord diagram with a single chord in a disk with exactly two boundary vertices. Assume that for any poset $\mathcal{P}$ satisfying conditions $i), ii), iii)$ with $\#\mathcal{P} = r$ for any positive integer $r < n+1$ there exists a chord diagram $d$ such that $\mathcal{P} = \mathcal{P}_d$. Let $\mathcal{Q}$ be a poset satisfying the above conditions where $\#\mathcal{Q} = n+1,$ and let $x \in \mathcal{Q}$ be a maximal element. We know $x$ covers either one or two elements of $\mathcal{Q}$.

Assume $x$ covers two elements $y,z \in \mathcal{Q}$. Since the Hasse diagram of $\mathcal{Q}$ has no cycles, we have that $\mathcal{Q} - \{x\} = \mathcal{Q}_1 + \mathcal{Q}_2$ where $y \in \mathcal{Q}_1$, $z \in \mathcal{Q}_2$, and $\mathcal{Q}_i$ satisfies $i), ii), iii)$ for $i \in [2]$. By induction, there exists positive integers $k_1, k_2$ satisfying $k_1 + k_2 = n$ and chord diagrams $d_i \in \mathcal{D}_{k_i,\epsilon^{(i)}}$ where $\mathcal{Q}_i = \mathcal{P}_{d_i}$ for $i \in [2]$ and where $\epsilon^{(i)} \in \{+,-\}^{k_i + 1}$ has all of its entries equal to the entries of $\epsilon$. By Lemma~\ref{rotationinv},  we can further assume that the chord corresponding to $y \in \mathcal{Q}_1$ (resp. $z \in \mathcal{Q}_2$)  is $c_1(i(y),k_1) \in d_1$ for some $i(y) \in [0,k_1-1]$ (resp.  {$c_2(j(z),k_2) \in d_2$} { for some $j(z) \in [1,k_2]$).} Define $d_1\sqcup d_2 := \{c^\prime(i^\prime_\ell,j^\prime_\ell)\}_{\ell \in [n]}$ to be the chord diagram in the disk with $n+2$ boundary vertices as follows (see Figure~\ref{disjunion}):
{$$\begin{array}{rclcc}
c^\prime(i^\prime_\ell, j^\prime_\ell) & := & \left\{\begin{array}{lcl} c_1(i_\ell, j_\ell) & : & \text{if $\ell \in [k_1]$}\\ \tau^{-(k_1+1)}c_2(i_{\ell - k_1},j_{\ell - k_1})  & : & \text{if $\ell \in [k_1 + 1, n]$}. \end{array}\right.
\end{array}$$}



\begin{figure}[h]
\includegraphics[scale=.85]{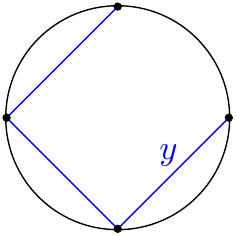} \raisebox{.35in}{$\text{ $\Large\sqcup$ }$} \includegraphics[scale=.85]{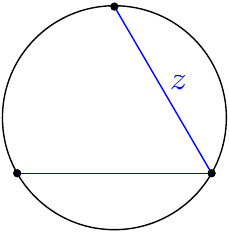} \raisebox{.35in}{ \text{ $\Large=$ } } \includegraphics[scale=.85]{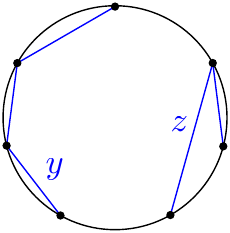}
\caption{An example with $k_1 = 3$ and $k_2 = 2$ so that $n = k_1 + k_2 = 5.$}
\label{disjunion}
\end{figure}

\noindent {Define $c^\prime(i^\prime_{n+1}, j^\prime_{n+1}): = c(k_1, n+1)$ }and then $d:= \{c^\prime(i^\prime_\ell, j^\prime_\ell)\}_{\ell \in [n+1]}$ satisfies $i), ii), iii),$ and $\mathcal{Q} = \mathcal{P}_d.$

Now assume $x$ covers only the element $y \in \mathcal{Q}$. In this case, the Hasse diagram of $\mathcal{Q}-\{x\}$ is connected. Now by induction the poset $\mathcal{Q} - \{x\} = \mathcal{P}_d$ for some diagram $d = \{c(i_\ell, j_\ell)\}_{\ell \in [n]} \in \mathcal{D}_{n,\epsilon}$ where we assume $i_\ell < j_\ell$. Let $y  = c(i(y),j(y)) \in \mathcal{Q} - \{x\}$ with $i(y) < j(y)$ denote the unique element that is covered by $x$ in $\mathcal{Q}$. This means that there are no chords in $d$ that are clockwise from $c(i(y), j(y))$. Without loss of generality, we assume that there are no chords in $d$ that are clockwise from $c(i(y), j(y))$ about $i(y)$.

We regard $d$ as an element of $\mathcal{D}_{n+1,\epsilon^\prime}$ by replacing it with $\widetilde{d}:= \{c^\prime(i^\prime_\ell, j^\prime_\ell)\}_{\ell \in [n]} \in \mathcal{D}_{n+1,\epsilon^\prime}$ as follows (see Figure~\ref{addpoint}): 

{$$\begin{array}{rclcc}
c^\prime(i^\prime_\ell, j^\prime_\ell) & := & \left\{\begin{array}{lcl} \rho^{-1}c(i_\ell, j_\ell) & : & \text{if $i_\ell \le i(y)$ and $j(y) \le j_\ell$,}\\  \tau^{-1}c(i_{\ell},j_{\ell})  & : & \text{if $j(y)  \le i_\ell$,}\\
c(i_\ell, j_\ell) & : & \text{otherwise}. \end{array}\right.
\end{array}$$}

\begin{figure}[h]
$$\begin{array}{ccccccccc}
\raisebox{.35in}{$d$} & \raisebox{.35in}{$=$} & \includegraphics[scale=.85]{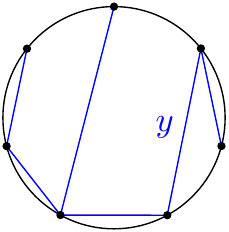} \raisebox{.35in}{ \text{ $\Large \longrightarrow$ } } & \raisebox{.35in}{$\widetilde{d}$} & \raisebox{.35in}{$=$} & \includegraphics[scale=.85]{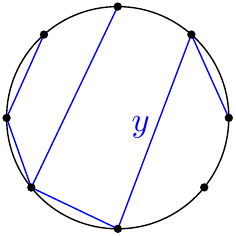}
\end{array}$$
\caption{An example with $n = 6$.}
\label{addpoint}
\end{figure}

\noindent Define $c^\prime(i^\prime_{n+1}, j^\prime_{n+1}) := c(i(y), i(y)+1)$ and put $d^\prime := \{c^\prime(i^\prime_\ell, j^\prime_\ell)\}_{\ell \in [n+1]}$. As $\mathcal{Q}-\{x\}$ satisfies $i), ii),$ and $iii)$, it is clear that the resulting chord diagram $d^\prime$ also satisfies $i), ii),$ and $iii)$, and that $\mathcal{P} = \mathcal{P}_{d^\prime}$. 
\end{proof}

Let $\mathcal{P}$ be a finite poset with $m = \#\mathcal{P}$. Let $f: \mathcal{P} \to \textbf{m}$ be an injective, order-preserving map (i.e. $x \le y$ implies $f(x) \le f(y)$ for all $x,y \in \mathcal{P}$) where $\textbf{m}$ is the linearly-ordered poset with $m$ elements. We call $f$ a \textbf{linear extension} of $\mathcal{P}$. We denote the set of linear extensions of $\mathcal{P}$ by $\mathscr{L}(\mathcal{P})$. Note that since $f$ is an injective map between sets of the same cardinality, $f$ is a bijective map between those sets.

\begin{theorem}\label{CESsandlinexts}
Let $d = \{c(i_\ell, j_\ell)\}_{\ell\in [n]} \in \mathcal{D}_{n, \epsilon}$ and let $\overline{\xi}_{\epsilon}$ denote the corresponding complete exceptional collection. Let $\text{CES}(\overline{\xi}_{\epsilon})$ denote the set of CESs that can be formed using only the representations appearing in $\overline{\xi}_{\epsilon}$. Then the map $\chi: \text{CES}(\overline{\xi}_{\epsilon}) \to \mathscr{L}(\mathcal{P}_d)$ defined by $(X_{i_1,j_1}^{\epsilon},\ldots, X_{i_n,j_n}^{\epsilon}) \stackrel{\chi_2}{\longmapsto} \{(c(i_\ell,j_\ell), n+1-\ell)\}_{\ell \in [n]} \stackrel{\chi_1}{\longmapsto} (f(c(i_\ell,j_\ell)) := n+1-\ell)$ is a bijection.
\end{theorem}

\begin{proof}
The map $\chi_2 = \Phi : \text{CES}(\overline{\xi}_{\epsilon}) \to \mathcal{D}_{n,\epsilon}(n)$ is a bijection by Theorem~\ref{ESbij}. Thus it is enough to prove that $\chi_1: \mathcal{D}_{n,\epsilon}(n) \to \mathscr{L}(\mathcal{P}_d)$ is a bijection.

First, we show that $\chi_1(d(n)) \in \mathscr{L}(\mathcal{P}_d)$ for any $d(n) \in \mathcal{D}_{n, \epsilon}(n)$. Let $d(n) \in \mathcal{D}_{n, \epsilon}(n)$ and let $f := \chi_1(d(n))$. Since the strand-labeling of $d(n)$ is good, if $(c_1, \ell_1)$ and $(c_2, \ell_2)$ are two labeled strands of $d(n)$ satisfying $c_1 \le c_2$, then $f(c_1) = \ell_1 \le \ell_2 = f(c_2).$ Thus $f$ is order-preserving. As the strands of $d(n)$ are bijectively labeled by $[n]$, we have that $f$ is bijective so $f \in \mathscr{L}(\mathcal{P}_d)$.

Next, define a map 
$$\begin{array}{rcl}
\mathscr{L}(\mathcal{P}_d) & \stackrel{\varphi}{\longrightarrow} & \mathcal{D}_{n,\epsilon}(n)\\
f & \longmapsto & \{(c(i_\ell, j_\ell), f(c(i_\ell, j_\ell)))\}_{\ell \in [n]}.
\end{array}$$
To see that $\varphi(f) \in \mathcal{D}_{n,\epsilon}(n)$ for any $f \in \mathscr{L}(\mathcal{P}_d)$, consider two labeled strands $(c_1, f(c_1))$ and $(c_2, f(c_2))$ belonging to $\varphi(f)$ where $c_1 \le c_2$. Since $f$ is order-preserving, $f(c_1) \le f(c_2).$ Thus the strand-labeling of $\varphi(f)$ is good so $\varphi(f) \in \mathscr{L}(\mathcal{P}_d)$.

Lastly, we have that $$\chi_1(\varphi(f)) = \chi_1(\{(c(i_\ell, j_\ell), f(c(i_\ell, j_\ell)))\}_{\ell \in [n]}) = f$$ and $$\varphi(\chi_1(\{(c(i_\ell, j_\ell), \ell)\}_{\ell \in [n]})) = \varphi(f(c(i_\ell, j_\ell)):= \ell) = \{(c(i_\ell, j_\ell), \ell)\}_{\ell \in [n]}$$ so $\varphi = \chi^{-1}_1$. Thus $\chi_1$ is a bijection.
\end{proof}

\section{Applications}
\label{sec:app}

Here we showcase some interesting results that follow easily from our main theorems.

\subsection{Labeled trees}
In \cite[p. 67]{sw86}, Stanton and White gave a nonpositive formula for the number of vertex-labeled trees with a fixed number of leaves. By connecting our work with that of Goulden and Yong \cite{gy02}, we obtain a positive expression for this number. Here we consider diagrams in $\mathcal{D}_{n,\epsilon}$ where $\epsilon = (-,\ldots,-)$ or $\epsilon = (+,\ldots,+)$. We regard these as chord diagrams to make clear the connection between our work and that of \cite{gy02}.
\begin{theorem}\label{trees}Let $T_{n+1}(r) := \{\text{trees on } [n+1] \text{ with } r \text{ leaves}\}$ and $\mathcal{D}_{n,\epsilon} := \{\text{diagrams } d = \{c(i_\ell,j_\ell)\}_{\ell \in [n]}\}$. Then
$$\begin{array}{rcl}
\#T_{n+1}(r) & = &\displaystyle \sum_{\begin{array}{c}\small\text{$d \in \mathcal{D}_{n,\epsilon}:$ \ $d \text{ has $r$ chords }c(i_j,i_j+1)$}\end{array}}\#\mathscr{L}(\mathcal{P}_d).
\end{array}$$ 
\end{theorem}
\begin{proof}
Observe that 
$$\begin{array}{rcl}
\displaystyle \sum_{\begin{array}{c}\small\text{$d \in \mathcal{D}_{n,\epsilon}$ : $d$ \text{ has $r$}}\\ \small \text{chords $c(i_j,i_j+1)$}\end{array}}\#\mathscr{L}(\mathcal{P}_d) & = & \displaystyle \sum_{\begin{array}{c}\small\text{$d \in \mathcal{D}_{n,\epsilon}: d \text{ has $r$}$}\\ \small \text{$\text{chords $c(i_j, i_j + 1)$}$}\end{array}}\#\{\text{good labelings of $d$}\} \\ 
& = & \#\left\{d(n) \in \mathcal{D}_{n,\epsilon}(n): \begin{array}{l}d(n) \text{ has $r$ chords $c(i_j,i_j + 1)$}\\ \text{for some $i_1,\ldots, i_r \in [0,n]$}\end{array}\right\}
\end{array}$$
where we consider $i_j + 1$ mod $n+1$. By \cite[Theorem 1.1]{gy02}, we have a bijection between diagrams $d \in \mathcal{D}_{n,\epsilon}$ with $r$ chords of the form $c(i_j, i_j + 1)$ for some $i_1, \ldots, i_r  \in [0,n]$ with good labelings and elements of $T_{n+1}(r)$. 
\end{proof}
\begin{corollary}
We have $(n+1)^{n-1} = \sum_{d \in \mathcal{D}_{n,\epsilon}} \#\mathscr{L}(\mathcal{P}_d)$.
\end{corollary}
\begin{proof}
Let $T_{n+1} := \{\text{trees on [n+1]}\}.$ One has that
$$\begin{array}{cclccc}
(n+1)^{n-1} & = & \#T_{n+1} \\
& = & \displaystyle\sum_{r \ge 0} \#T_{n+1}(r) \\
& = & \displaystyle\sum_{r \ge 0}  \sum_{\begin{array}{c}\small\text{$d \in \mathcal{D}_{n,\epsilon}$ : $d$ \text{ has $r$}}\\ \small \text{chords $c(i_j,i_j+1)$}\end{array}}\#\mathscr{L}(\mathcal{P}_d) & \text{(by Theorem~\ref{trees})} \\
& = & \displaystyle \sum_{d \in \mathcal{D}_{n,\epsilon}} \#\mathscr{L}(\mathcal{P}_d).
\end{array}$$\end{proof}

\subsection{Reddening sequences}
In \cite{k12}, Keller proves that for any quiver $Q$, any two reddening mutation sequences applied to $\widehat{Q}$ produce isomorphic ice quivers.  As mentioned in \cite{kel13}, his proof is highly dependent on representation theory and geometry, but the statement is purely combinatorial---we give a combinatorial proof of this result for type $\mathbb{A}_n$ quivers $Q_\epsilon$.


Let $R \in EG(\widehat{Q})$.  A mutable vertex {$i \in R_0$} is called \textbf{green} if there are no arrows $j \to i$ in $R$ with $j \in [n+1,m]$.  Otherwise, $i$ is called \textbf{red}.  A sequence of mutations $\mu_{i_r}\circ \cdots \circ \mu_{i_1}$ is \textbf{reddening} if all  mutable vertices of the quiver $\mu_{i_r}\circ \cdots \circ \mu_{i_1}(\widehat{Q})$ are red.  Recall that an isomorphism of quivers that fixes the frozen vertices is called a \textbf{frozen isomorphism}.  We now state the theorem.
\begin{theorem}
If $\mu_{i_r}\circ \cdots \circ \mu_{i_1}$ and $\mu_{j_s}\circ \cdots \circ \mu_{j_1}$ are two reddening sequences of $\widehat{Q}_\epsilon$ for some $\epsilon \in \{+,-\}^{n+1}$, then there is a frozen isomorphism $\mu_{i_r}\circ \cdots \circ \mu_{i_1}(\widehat{Q}_{\epsilon}) \cong \mu_{j_s}\circ \cdots \circ \mu_{j_1}(\widehat{Q}_{\epsilon})$.
\end{theorem}
\begin{proof}
Let $\mu_{i_r}\circ \cdots \circ \mu_{i_1}$ be any reddening sequence.  Denote by $C$ the $\bc$-matrix of $\mu_{i_r}\circ \cdots \circ \mu_{i_1}(\widehat{Q}_{\epsilon})$.  By Theorem~\ref{c-matClassif}, $C$ corresponds to an oriented strand diagram  $\overrightarrow{d}_C \in \overrightarrow{\mathcal{D}}_{n,\epsilon}$ with all strands of the form $\overrightarrow{c}(j,i)$ for some $i$ and $j$ satisfying $i < j$.  As $\overrightarrow{d}_C$ avoids the configurations described in Definition~\ref{def:cmatdiag}, we conclude that $\overrightarrow{d}_C = \{\overrightarrow{c}(i, i-1)\}_{i\in[n]}$ and $C = -I_n$.  Since \textbf{c}-matrices are in bijection with ice quivers in $EG(\widehat{Q}_\epsilon)$ (see \cite[Thm 1.2]{nz12}) and since $\widecheck{Q}_\epsilon$ is an ice quiver in $EG(\widehat{Q}_\epsilon)$ whose \textbf{c}-matrix is $-I_n$, we obtain the desired result.
\end{proof}

\subsection{Noncrossing partitions and exceptional sequences}
In this section, we give a combinatorial proof of Ingalls' and Thomas' result that complete exceptional sequences are in bijection with maximal chains in the lattice of noncrossing partitions \cite{it09}. We remark that their result is more general than that which we present here. Throughout this section, we assume that $Q_\epsilon$ has $\epsilon= (-,\ldots,-)$ and we regard the strand diagrams of $Q_\epsilon$ as chord diagrams.

A \textbf{partition} of $[n]$ is a collection $\pi = \{B_\alpha\}_{\alpha \in I} \in 2^{[n]}$ of subsets of $[n]$ called \textbf{blocks} that are nonempty, pairwise disjoint, and whose union is $[n].$ We denote the lattice of set partitions of $[n]$, ordered by refinement, by $\Pi_n$. A set partition $\pi = \{B_{\alpha}\}_{\alpha \in I} \in \Pi_n$ is called \textbf{noncrossing} if for any $i < j < k < \ell$ where $i, k \in B_{\alpha_1}$ and $j, \ell \in B_{\alpha_2}$, one has $B_{\alpha_1} = B_{\alpha_2}.$ We denote the lattice of noncrossing partitions of $[n]$ by  $NC^{\mathbb{A}}(n)$.

Label the vertices of a convex $n$-gon $\mathcal{S}$ with elements of $[n]$ so that reading the vertices of $\mathcal{S}$ counterclockwise determines an increasing sequence mod $n$. We can thus regard $\pi = \{B_\alpha\}_{\alpha \in I} \in NC^\mathbb{A}(n)$ as a collection of convex hulls $B_\alpha$ of vertices of $\mathcal{S}$ where $B_\alpha$ has empty intersection with any other block $B_{\alpha^\prime}$.

Let $n = 5$. The following partitions all belong to $\Pi_5$, but only $\pi_1, \pi_2, \pi_3 \in NC^\mathbb{A}(5).$ $$\pi_1 = \{\{1\}, \{2,4,5\}, \{3\}\}, \pi_2 = \{\{1,4\}, \{2,3\}, \{5\}\}, \pi_3 = \{\{1,2,3\}, \{4,5\}\}, \pi_4 = \{\{1,3,4\}, \{2,5\}\}$$  Below we represent the partitions $\pi_1, \ldots, \pi_4$ as convex hulls of sets of vertices of a convex pentagon.  We see from this representation that $\pi_4 \not \in NC^\mathbb{A}(5).$

\begin{center}
\includegraphics[scale=.75]{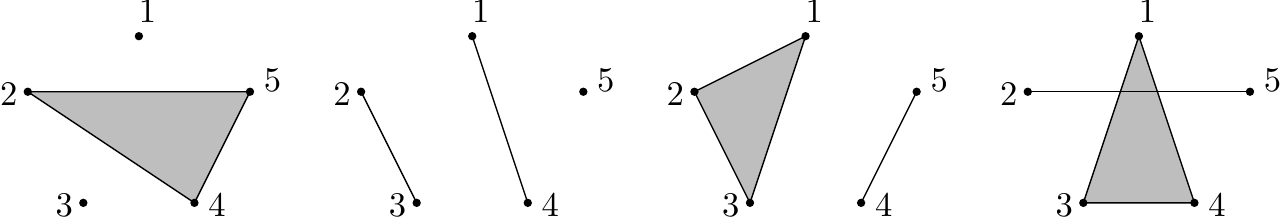}
\end{center}

\begin{theorem}\label{chains}
Let $k \in [n].$ There is a bijection between $\mathcal{D}_{k,\epsilon}(k)$ and the following chains in $NC^{\AAA}(n+1)$ \[\left\{(\pi_1 = \{\{i\}\}_{i \in [n+1]}, \pi_2, \ldots, \pi_{k+1}) \in (NC^\mathbb{A}(n+1))^{k+1}: \begin{array}{c}\pi_j = (\pi_{j-1} \backslash\{B_{\alpha}, B_\beta\})\sqcup \{B_\alpha \sqcup B_\beta\}\\ \text{for some $B_\alpha \neq B_\beta$ in $\pi_{j-1}$}\end{array}\right\}.\] In particular, when $k = n$, there is a bijection between $\mathcal{D}_{n,\epsilon}(n)$ and maximal chains in $NC^{\mathbb{A}}(n+1)$. We remark that each chain described above is \textbf{saturated} (i.e. each inequality appearing in $\{\{i\}\}_{i \in [n+1]}  < \pi_1 < \cdots < \pi_{k}$ is a covering relation). 
\end{theorem}

\begin{proof}
Let $d(k) = \{(c(i_\ell,j_\ell),\ell)\}_{\ell \in [k]} \in \mathcal{D}_{k,\epsilon}(k)$. Define $\pi_{d(k),1} := \{\{i\}\}_{i \in [n+1]} \in \Pi_{n+1}.$ Next, define $\pi_{d(k),2} := \left(\pi_{d(k),1}\backslash \{\{i_1 + 1\}, \{j_1+1\}\}\right) \sqcup \{\{i_1+1,j_1+1\}\}$. Now assume that $\pi_{d(k),s}$ has been defined for some $s \in [k]$. Define $\pi_{d(k),s+1}$ to be the partition obtained by merging the blocks of $\pi_{d(k),s}$ containing $i_{s}+1$ and $j_{s}+1$. Now define $f(d(k)) := (\pi_{d(k),1}, \ldots, \pi_{d(k), k+1}).$

It is clear that $f(d(k))$ is a chain in $\Pi_{n+1}$ with the desired property as $\pi_1 \lessdot \pi_2$ in $\Pi_{n+1}$ if and only if $\pi_2$ is obtained from $\pi_1$ by merging exactly two distinct blocks of $\pi_1.$ To see that each $\pi_{d(k),s} \in NC^\mathbb{A}(n+1)$, suppose a crossing of two blocks occurs in a partition appearing in $f(d(k))$. Let $\pi_{d(k),s}$ be the smallest partition of $f(d(k))$ (with respect to the partial order on set partitions) with two blocks crossing blocks $B_1$ and $B_2$. Without loss of generality, we assume that $B_2 \in \pi_{d(k),s}$ is obtained by merging the blocks $B_{\alpha_1}, B_{\alpha_2} \in \pi_{d(k),s-1}$ containing $i_{s-1} + 1$ and $j_{s-1}+1$, respectively. This means that $c(i_{s-1},j_{s-1}) \in d(k)$ and $c(i_{s-1}, j_{s-1})$  crosses at least one other chord of $d(k)$. This contradicts that $d(k) \in \mathcal{D}_{k,\epsilon}(k)$. Thus $f(d(k))$ is a chain in $NC^\mathbb{A}(n+1)$ with the desired property.

Next, we define a map $g$ that is the inverse of $f$. Let $C = (\pi_1 = \{\{i\}\}_{i \in [n+1]}, \pi_2, \cdots, \pi_{k+1}) \in (NC^\mathbb{A}(n+1))^{k+1}$ be a chain where each partition in $C$ satisfies $\pi_j = (\pi_{j-1}\backslash\{B_\alpha, B_\beta\}) \sqcup \{B_\alpha \sqcup B_\beta\}$ for some $B_\alpha \neq B_\beta$ in $\pi_{j-1}$. As $\pi_2 = \left(\pi_1 \backslash \{\{s_1\}, \{t_1\}\}\right)\sqcup \{\{s_1, t_1\}\}$, define $c(i_1,j_1) := c(s_1-1, t_1-1)$ where we consider $s_1 - 1$ and $t_1-1$ mod $n+1.$ Now for $r \ge 2$ let $B_1, B_2 \in \pi_{r-1}$ be the blocks that one merges to obtain $\pi_r$. Define $s_1 \in B_1$ (resp., $t_1 \in B_1$) to be the last element of $B_1$ (resp., $B_2$) that one encounters before any element of $B_2$ (resp., $B_1$) while reading counterclockwise through the integers $1, \ldots, n$. Let $c(i_{r-1},j_{r-1}) := c(s_1-1,t_2-1)$. Finally, put $g(C) := \{(c(i_\ell,j_\ell),\ell): \ell \in [k]\}$.

We claim that $g(C)$ has no crossing chords. Suppose $(c(s_i,t_i), i)$ and $(c(s_j,t_j),j)$ are crossing chords in $g(C)$ with $i < j$ and $i,j \in [k]$. We further assume that $$j = \min\{j^\prime \in [i+1,k]: \text{$(c(s_{j^\prime},t_{j^\prime}), j^\prime)$ crosses $(c(s_i,t_i),i)$ in $g(C)$}\}.$$
We observe that $s_i + 1, t_i + 1 \in B_1$ for some block $B_1 \in \pi_{j}$ and that $s_j+1, t_j+1 \in B_2$ for some block $B_2 \in \pi_{j+1}$. We further observe that $s_j+1, t_j+1 \not \in B_1$ otherwise, by the definition of the map $g$, the chords $(c(s_i,t_i),i)$ and $(c(s_j,t_j),j)$ would be noncrossing. Thus $B_1, B_2 \in \pi_{j+1}$ are distinct blocks that cross so $\pi_{j+1} \not \in NC^\mathbb{A}(n+1).$ We conclude that $g(C)$ has no crossing chords so $g(C) \in \mathcal{D}_{k,\epsilon}(k)$.

To complete the proof, we show that $g\circ f = 1_{\mathcal{D}_{k,\epsilon}(k)}$. The proof that $f \circ g$ is the identity map is similar. Let $d(k) \in \mathcal{D}_{k,\epsilon}(k)$. Then $f(d(k)) = (\pi_1 = \{\{i\}\}_{i \in [n+1]}, \pi_2, \ldots, \pi_{k+1})$ where for any $s \in [k]$ we have $$\pi_s = \left(\pi_{s-1}\backslash\{B_\alpha, B_\beta\}\right) \sqcup \{B_\alpha, B_\beta\}$$ where $i_{s-1} + 1 \in B_\alpha$ and $j_{s-1} + 1 \in B_\beta.$ Then we have $g(f(d(k)))  =  \{c((i_\ell +1) - 1, (j_\ell +1) - 1), \ell)\}_{\ell \in [k]} = \{(c(i_\ell, j_\ell), \ell)\}_{\ell \in [k]}.$\end{proof}

\begin{corollary}
If $\epsilon = (-,\ldots, -)$ of $\epsilon = (+, \ldots, +)$, then the exceptional sequences of $Q_\epsilon$ are in bijection with saturated chains in $NC^\mathbb{A}(n+1)$ of the form \[\left\{(\pi_1\{\{i\}\}_{i \in [n+1]}, \pi_2, \ldots, \pi_{k+1}) \in (NC^\mathbb{A}(n+1))^{k+1}: \begin{array}{c}\pi_j = (\pi_{j-1} \backslash\{B_{\alpha}, B_\beta\})\sqcup \{B_\alpha \sqcup B_\beta\}\\ \text{for some $B_\alpha \neq B_\beta$ in $\pi_{j-1}$}\end{array}\right\}.\]
\end{corollary}

\begin{figure}
$$\begin{array}{ccc}
\includegraphics[scale=.7]{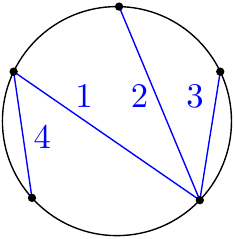} & \raisebox{.3in}{$\longmapsto$} & \raisebox{.1in}{\includegraphics[scale=.7]{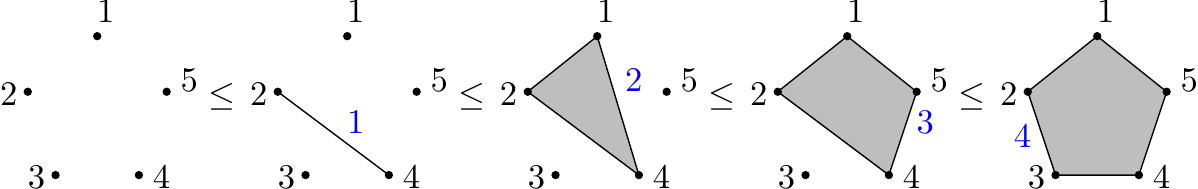}}\\
\includegraphics[scale=.7]{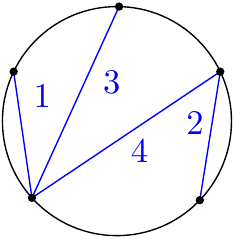} & \raisebox{.3in}{$\longmapsto$} & \raisebox{.1in}{\includegraphics[scale=.7]{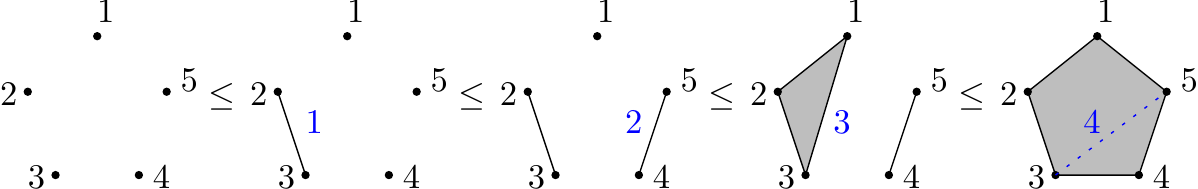}}
\end{array}$$
\caption{Two labeled strand diagrams and their corresponding maximal chains in $NC^{\mathbb{A}}(5).$}
\label{fig_max_chain_bij}
\end{figure}

\begin{example}
In Figure~\ref{fig_max_chain_bij}, we give two examples of the bijection from Theorem~\ref{chains} with $k = 4$.
\end{example}

\newpage

\nocite{*}
\bibliographystyle{alpha}
\bibliography{refs}
\label{sec:biblio}

\end{document}